\documentclass[11pt,reqno]{amsart}
\usepackage{amsfonts,amssymb,amscd,amsmath,enumerate,verbatim,calc,graphicx}
\usepackage{amsmath}
\usepackage{amssymb}
\usepackage{amscd}
\usepackage{wrapfig}
\usepackage{float}
\usepackage{color}
\usepackage{mathtools}
\usepackage[colorlinks,pagebackref=true]{hyperref}
\makeatletter

\makeatletter
\def\widebreve#1{\mathop{\vbox{\m@th\ialign{##\crcr\noalign{\kern3\p@}%
				\brevefill\crcr\noalign{\kern3\p@\nointerlineskip}%
				$\hfil\displaystyle{#1}\hfil$\crcr}}}\limits}

\def\brevefill{$\m@th \setbox\z@\hbox{$\braceld$}%
	\bracelu\leaders\vrule \@height\ht\z@ \@depth\z@\hfill\braceru$}
\makeatletter

\def\@citecolor{blue}
\def\@linkcolor{blue}
\def\@urlcolor{blue}
\def\@urlcolor{blue}

\def\NZQ{\mathbb}               
\def\NN{{\NZQ N}}
\def\QQ{{\NZQ Q}}
\def\ZZ{{\NZQ Z}}

\def\mfp{\mathfrak p}
\def\mfq{\mathfrak q}

\def\pol{\operatorname{pol}}
\def\bight{\operatorname{bight}}

\def \mm{{\mathfrak m}}
\def\Ass{\operatorname{Ass}}
\def\svd{\operatorname{svd}}
\def\MinAss{\operatorname{MinAss}}
\def\height{\operatorname{height}}

\def\deg{\operatorname{degree}}

\def\reg{\operatorname{reg}}
\newtheorem{Theorem}{Theorem}[section]
\newtheorem{Lemma}[Theorem]{Lemma}
\newtheorem{Corollary}[Theorem]{Corollary}

\newtheorem{Remark}[Theorem]{Remark}

\newtheorem{Example}[Theorem]{Example}

\newtheorem{Question}[Theorem]{Question}
\let\epsilon\varepsilon
\let\phi=\varphi
\let\kappa=\varkappa

\def \s {\sigma}

\def \om {\omega}

\begin{document}
	
	\title{$v$-numbers of symbolic power filtrations}
	
	\author{Vanmathi A}
	\address{Vanmathi A, 
		Department of Mathematics,
		Indian Institute of Technology, Palakkad, India}
	\email{vanmathianandarajan@gmail.com, 212404007@smail.iitpkd.ac.in}
	
	\author{Parangama Sarkar}
	\thanks{The second author was partially supported by SERB POWER Grant with Grant No. SPG/2021/002423.}
	\address{Parangama Sarkar,  Department of Mathematics,
		Indian Institute of Technology, Palakkad, India}
	\email{parangamasarkar@gmail.com, parangama@iitpkd.ac.in}
	
	\begin{abstract}
		We study the asymptotic behaviour of  $v$-number and local $v$-numbers of Noetherian generalized symbolic power filtrations $\mathcal I=\{I_n\}$ in a  Noetherian $\NN$-graded domain and show that they are quasi-linear type. We provide sufficient conditions for the existence of the limits $\lim\limits_{n\to\infty}\frac{v(I_n)}{n}$ and $\lim\limits_{n\to\infty}\frac{v_\mfp(I_n)}{n}$ for all $\mfp\in\overline A(\mathcal I)$.  We explicitly compute local $v$-numbers and $v$-numbers of symbolic powers of cover ideals of complete bipartite graphs, complete graphs, cycles, $K_m^s$ and compare them with their Castelnuovo-Mumford regularity. For every positive integer $p\geq 2$, we provide an example of an unmixed bipartite graph $\mathcal H_p$ that is not a complete multipartite graph and $v(J(\mathcal H_p))\geq\bight(I(\mathcal H_p))$. This answers a question in \cite[Question 3.12]{KSa}. We show that 
		for both connected bipartite graphs and  connected non-bipartite graphs, the difference between the regularity and the $v$-number of the cover ideals can be arbitrarily large. This strengthens and gives an alternative proof of \cite[Theorem 3.10]{KSa}. 
	\end{abstract}

	\maketitle
	\section{Introduction}
		Let $R$ be an $\NN$-graded Noetherian ring and $I$ be a homogeneous ideal of $R$. It is well-known that for any $\mfp\in\Ass(I)$, there exists a homogeneous element $f\in R$ such that $\mfp=(I:f)$. The invariant, $v$-number of $I$, defined as
		$$v(I): =\min\{u: \mbox{there exists } f\in R_u \mbox{ and }\mfp\in \Ass(I)\mbox{ such that }\mfp=(I:f)\},$$ was first introduced in \cite{CSTPV} to study the asymptotic behaviour of the minimum distance of projective Reed–Muller-type codes.
	For any $\mfp\in \Ass(I)$, the local $v$-number of $I$ at $\mfp$ is defined in the following way \cite{CSTPV},
	$$v_\mfp(I):=\min\{u: \mbox{there exists } f\in R_u \mbox{ such that }\mfp=(I:f) \}$$ and from the definition it follows that $v(I)=\min\{v_\mfp(I): \mfp\in\Ass(I)\}$. Therefore the local $v$-numbers are useful in computing the $v$-number explicitly. $v$-numbers are explored further for various homogeneous ideals and linked with the Castelnuovo-Mumford regularity of those ideals   (see \cite{GRV}, \cite{FS}, \cite{JV}, \cite{JVS},  \cite{KSa}, \cite{SS}, \cite{F23} and the references therein). Recently there has been growing interest in studying the asymptotic behaviour of these $v$-invariants (see \cite{FS}, \cite{Co}, \cite{GF}, \cite{BMS}, \cite{FSmarch} \cite{MRK}). In \cite{Br}, Brodmann proved that $\Ass(R/ {I^n})=\Ass(R/{I^{n+1}})$ for all $n\gg 0$. Let $\overline A(I)$ denote $\Ass(R/ {I^n})$ for all $n\gg 0$. In \cite{Co}, it is shown that for all $\mfp\in\overline A(I)$, the invariants $v_\mfp(I^n)$ are eventually linear polynomials in $n$ with integer coefficients and hence $v(I^n)$ is an eventually linear polynomial in $n$.

	The objective of this paper is to investigate the asymptotic behaviour of $v$-numbers and local $v$-numbers for filtrations $\mathcal I=\{I_n\}$ of ideals. The following examples show that, in general for filtrations, these invariants may not be linear polynomials in $n$ with integer coefficients for $n\gg 0$ and $\displaystyle\lim\limits_{n\to\infty}\frac{v_\mfp(I_n)}{n}$ may not exist. These two examples answer the question \cite[Question 5.2]{FSmarch}.
	\begin{Example}{\em	
		Consider the filtration $\mathcal I=\{I_n=(x^2,xy^{n^2})\}$ of homogeneous ideals in the polynomial ring $k[x,y]$ where $k$ is a field. Then $\mathcal I$ is a non-Noetherian filtration. Note that $\Ass(R/I_n)=\{\mm=(x,y),\mfp=(x)\}$ for all $n\geq 1$. For all $n\geq 1$, by  \cite[Lemma 1.2]{Co}, $$\displaystyle v_\mm({I_n})=\min\{w: \Big(\frac{I_{n}:\mm}{I_{n}}\Big)_w\neq 0\}=n^2=\min\{w: \Big(\frac{I_{n}:\mfp}{(I_{n}:\mfp)\cap(I_{n}:\mm^\infty)}\Big)_w\neq 0\}=v_\mfp({I_n}).$$
 Hence $v_\mm({I_n})$ and $v_\mfp({I_n})$ are not linear polynomials or quasi-linear polynomials for all $n\gg 0$ and $\displaystyle\lim\limits_{n\to\infty}\frac{v_\mm(I_n)}{n}$, $\displaystyle\lim\limits_{n\to\infty}\frac{v_\mfp(I_n)}{n}$ do not exist.}
	\end{Example}	
	\begin{Example}{\em	
		Consider the filtration $\mathcal I=\{I_n=(x^{\lceil{n\sqrt{2}}\rceil})\}$ of homogeneous ideals in the polynomial ring $k[x]$ where $k$ is a field. Then $\mathcal I$ is a non-Noetherian filtration. Note that $ \Ass(R/I_n)=\{\mm=(x)\}$ for all $n\geq 1$. For all $n\geq 1$, by  \cite[Lemma 1.2]{Co}, $$\displaystyle v_\mm({I_n})=\min\{w: \Big(\frac{I_{n}:\mm}{I_{n}}\Big)_w\neq 0\}=\lceil{n\sqrt{2}}\rceil-1.$$ Hence for all $n\gg 0$,  $v_\mm({I_n})$  is not a linear polynomial with integer leading term but $\displaystyle\lim\limits_{n\to\infty}\frac{v_\mm(I_n)}{n}=\sqrt{2}$.}
	\end{Example}	
In this paper, we primarily focus on integral closure filtrations of ideals and generalized symbolic power filtrations.  In Remark \ref{Noetherian}, we mention some instances where it is known that these filtrations are Noetherian filtrations. In Section \ref{def}, we discuss and recall the required notation and definitions. The main results in Section \ref{section3}  is summarized in the following theorem.
\begin{Theorem}{\em(Theorem \ref{integralclosure} and Theorem \ref{main}) } 
Let $R$ be an $\mathbb N$-graded Noetherian domain.
\begin{enumerate}
		\item[$(1)$] Let $I$ be a homogeneous ideal of $R$ generated by homogeneous elements $a_1,\ldots,a_l$ with $\deg a_i=d_i$ for all $1\leq i\leq l$. Consider the filtration $\mathcal I=\{I_n=\overline{I^n}\}$.  Suppose the $R$-algebra $\bigoplus_{n\in\mathbb N}\overline{I^n}$ is a finitely generated $R[I]$-module. Let $\mfp\in\overline A(\mathcal I)$. Then $v_\mfp(\overline{I^n})$ and $v({\overline{I^n}})$  are eventually linear functions in $n$ with leading coefficient from the set $\{d_1,\ldots,d_l\}$. 	
\item[$(2)$] Consider a generalized symbolic power filtration $\mathcal I=\{I_n=(I^n:L^\infty)\}$ where  $I$ and $L$  are two homogeneous ideals of $R$. Suppose $\mathcal I$ is a Noetherian filtration. Then the following hold.
	\begin{enumerate}
		\item[$(a)$] $v_\mfp({I_n})$, $v({I_n})$  are  quasi-linear type for all $\mfp\in\overline A(\mathcal I)$.
		\item[$(b)$] If one the following holds
		\begin{enumerate}
			\item  $R=K[x_1,\ldots,x_m]$ is a polynomial ring over a field $K$ and $I$ is an equigenerated monomial ideal,
			\item $I$ is an equigenerated homogeneous ideal and $R_0$ is a local ring with infinite residue field,
		\end{enumerate}
	then $ \displaystyle\lim\limits_{n\to\infty}\frac{v_\mfp(I_n)}{n}$ exist for all $\mfp\in\overline A(\mathcal I)$ and $\displaystyle\lim\limits_{n\to\infty}\frac{v(I_n)}{n}=\frac{\alpha(I_{\svd(\mathcal I)})}{\svd(\mathcal I)}$ where $\svd(\mathcal I)$ is the standard Veronese degree of $\mathcal I$.
		\item[$(c)$] Let $R=K[x_1,\ldots,x_m]$ be a polynomial ring over a field $K$, $J$ be a cover ideal of a finite simple graph $G$ with $m$ vertices and $\mathcal I=\{J^{(n)}\}$. Then  
		
		$\displaystyle\lim\limits_{n\to\infty}\frac{v(I_n)}{n}, \displaystyle\lim\limits_{n\to\infty}\frac{v_\mfp(I_n)}{n}$ exist for all $\mfp\in\overline A(\mathcal I)$ and $\displaystyle\lim\limits_{n\to\infty}\frac{v(I_n)}{n}=\frac{\alpha(J^{(2)})}{2}$.		
	\end{enumerate}	
\end{enumerate}	
\end{Theorem}	

In section \ref{coverideal}, we explore the relation between $v$-numbers and Castelnuovo-Mumford regularity of symbolic power filtrations. We explicitly compute local $v$-numbers and $v$-numbers of symbolic powers of cover ideals of complete bipartite graphs, complete graphs, cycles and $K_m^s$ (Theorem \ref{combi}, Theorem \ref{complete}, Theorem \ref{cycle} and Theorem \ref{corona}). We compare these $v$-numbers with the Castelnuovo-Mumford regularity of the associated ideals. As a consequence of Theorem \ref{combi}, we answer a question raised by Saha in \cite[Question 3.12]{KSa}.
\begin{Corollary}{\em(Corollary \ref{answer})}
		For every positive integer $p\geq 2$, there exists an unmixed bipartite graph $\mathcal H_p$ which is not a complete multipartite graph and $v(J(\mathcal H_p))-\bight(I(\mathcal H_p))=p-2$ where $\bight(I(\mathcal H_p))=\max\{\height(Q):Q\in\Ass(I(\mathcal H_p))\}$. In particular, $v(J(\mathcal H_p))\geq\bight(I(\mathcal H_p))$.
\end{Corollary}	
One of the main results of Section \ref{coverideal} is the following which shows that
for both connected bipartite graphs and  connected non-bipartite graphs, the difference between the regularity and the $v$-number of the cover ideal can be arbitrarily large. This strengthens and gives an alternative proof of \cite[Theorem 3.10]{KSa} (note that the example in \cite[Theorem 3.10]{KSa} is a non-bipartite graph). 
\begin{Theorem}{\em(Theorem \ref{arbitrary})}
		For every positive integer $k$, 
		\begin{enumerate}
	\item[$(1)$] there exist infinitely many connected bipartite graphs $\mathcal H_k$ with the cover ideal $J(\mathcal H_k)$ such that $\reg(S/J(\mathcal H_k))-v(J(\mathcal H_k))=k$ where $S$ is the polynomial ring over a field $K$ with $|V(\mathcal H_k)|$ variables.
	\item[$(2)$] there exists a connected non-bipartite graph $\mathcal H_k$ with the cover ideal $J(\mathcal H_k)$ such that $\reg(S/J(\mathcal H_k))-v(J(\mathcal H_k))=k$ where $S$ is the polynomial ring over a field $K$ with $|V(\mathcal H_k)|$ variables.
	\end{enumerate}
	\end{Theorem}	
We provide classes of graphs whose symbolic powers of cover ideals satisfy the following:
\begin{enumerate}
	\item $\reg(R/J(G)^{(n)})=v(J(G)^{(n)})$ for all $n\geq 1$ {(Theorem \ref{combi})},
\item $\reg(R/J(G)^{(n)})-v(J(G)^{(n)})=n-1$ for all $n\geq 1$ {(Theorem \ref{combi})}, 
\item	$\reg(R/J(G)^{(n)})=v(J(G)^{(n)})$ for $n=1,2$ and $\reg(R/J(G)^{(n)})>v(J(G)^{(n)})$ for all $n\geq 3$ {(Theorem \ref{complete})}.
\end{enumerate}	 
Motivated by Theorem \ref{combi} and Theorem \ref{complete}, we pose the following question.
\begin{Question}{\em(Question \ref{question})}	
	For every  integer $m>2$, does there exist a graph $G_m$ with the cover ideal $J_m$ such that $\reg(R/J_m^{(n)})=v(J_m^{(n)})$ for all $n\leq m$ and $\reg(R/J_m^{(n)})\neq v(J_m^{(n)})$ for all $n>m$?
\end{Question}			


We conclude Section \ref{coverideal}  by showing that for Noetherian filtrations the difference between the leading coefficients of local $v$-numbers, i.e., $\displaystyle\lim\limits_{n\to \infty}\frac{v_Q(I_n)}{n}- \lim\limits_{n\to \infty}\frac{v_P(I_n)}{n}$ can be arbitrary large and this helps us to provide a counterexample to \cite[Conjecture 5.4]{FSmarch}.
\begin{Corollary}{\em(Corollary \ref{counterexample})}	
For every positive integer $t$,
\begin{enumerate}
	\item[$(1)$] 	there exists a Noetherian filtration $\{I_n\}$ in a Noetherian $\mathbb  N$-graded domain such that $\Ass(I_n)=\Ass(I_{n+1})$ for all $n\gg0$ and there exist two primes $P,Q$ which are maximal in $\Ass(I_n)$ for all $n\gg0$ such that $\displaystyle\lim\limits_{n\to \infty}\frac{v_Q(I_n)}{n}- \lim\limits_{n\to \infty}\frac{v_P(I_n)}{n}=t$.	
	\item[$(2)$]	there exists a homogeneous ideal $I_t$ in a Noetherian $\mathbb  N$-graded domain and there exist two primes $P,Q$ which are maximal in $\overline A(I_t)$ such that $\displaystyle\lim\limits_{n\to \infty}\frac{v_P(I_t^n)}{n}\neq \lim\limits_{n\to \infty}\frac{v_Q(I_t^n)}{n}$.
\end{enumerate}			
\end{Corollary}	 
		\section{Notation and definitions}\label{def}
	Let $I$ be a homogeneous ideal in a Noetherian $\mathbb N$-graded ring. We define $\alpha(I)=\min\{d: I_d\neq 0\}$. We say, $I$ is equigenerated if $I$ is generated by some homogenous elements of the same degree. 
		
		Let $M$ be a finitely generated graded module over a polynomial ring $R=k[x_1,\ldots,x_n]$ where $k$ is a field. Consider a  minimal graded free resolution of $M$ as $R$-module
		$$0\rightarrow \bigoplus\limits_{j}R(-j)^{\beta_{pj}}\rightarrow \bigoplus\limits_{j}R(-j)^{\beta_{(p-1)j}}\rightarrow\cdots\rightarrow\bigoplus\limits_{j}R(-j)^{\beta_{1j}}\rightarrow \bigoplus\limits_{j}R(-j)^{\beta_{0j}}\rightarrow 0.$$ Then the Castelnuovo-Mumford regularity of $M$, denote by $\reg(M)$, is defined as $$\reg(M):=\max\{j-i: \beta_{ij}\neq 0, 0\leq i\leq p\}.$$  For any homogeneous ideal $I$ in a Noetherian $\mathbb N$-graded ring $R$, we have $\reg(R/I)=\reg(I)-1$.
		
		A function $f:\NN\rightarrow \NN$ is  called quasi-linear type of period $t$ for some integer $t\geq 1$ if there exist linear polynomials $P_0(X),\ldots, P_{t-1}(X)\in\QQ[X]$ such that for all $n\gg 0$, we have $f(n)=P_i(n)$ where $n=mt+i$ with $0\leq i\leq t-1$.
		
	A collection $\mathcal I=\{I_n\}_{n\in\NN}$ of ideals  in a commutative Noetherian ring $R$ is called a filtration if $I_0=R$, for all $r,s\in \NN$, $I_rI_s\subset I_{r+s}$  and
	$ I_s\subset I_r$ whenever $r\leq s$. A filtration $\mathcal I=\{I_n\}$ of ideals in a ring $R$ is called a Noetherian filtration if $\bigoplus_{n\ge 0}I_n$ is a finitely generated $R$-algebra. For an ideal $I$ in $R$, we denote $\bigoplus_{n\ge 0}I^n$ by $R[I]$. 
	
	Suppose $\mathcal I=\{I_n\}$ is a filtration of ideals such that $\Ass(R/ {I_n})=\Ass(R/{I_{n+1}})$ for all $n\gg 0$. Then for all $n\gg 0$, we denote the set $\Ass(R/ {I_n})$ by $\overline A(\mathcal I)$. If $\mathcal I=\{I^n\}$ then we denote $\overline A(\mathcal I)$ by $\overline A(I)$.
	
	For an ideal $I$ in $R$, the $n$-th symbolic power of $I$ is defined in two ways in literature:
	\begin{enumerate}
		\item $I^{(n)}=\bigcap\limits_{Q\in\MinAss(I)}I^nR_Q\cap R$,
		\item 	$I^{(n)}=\bigcap\limits_{Q\in\Ass(I)}I^nR_Q\cap R$.	
	\end{enumerate}	
	In both cases, there exists an ideal $J$ (if $I$ is homogeneous (or monomial) then there exists a homogeneous (or monomial) ideal $J$) such that $I^{(n)}=I^n:J^{\infty}$ for all $n\geq 1$ (\cite[Lemma 2.8 and Lemma 2.9]{HJKN}). 
\subsection{Graphs}	Let $G$ be a finite simple graph with the vertex set $V(G)=\{1,\ldots,u\}$ and the edge set $E(G)$. Then the edge ideal of $G$  is $$I(G)=\langle x_i x_j : (i,j)\in E(G)\rangle$$ and the cover ideal of $G$ is $J(G)=\displaystyle\bigcap\limits_{(i,j)\in E(G)}\langle {x_i,x_j}\rangle$ in $K[x_1,\ldots,x_u]$ where $K$ is a field.  A subset $S$ of $V(G)$ is called a vertex cover of $G$ if $\{i,j\}\cap S\neq\emptyset$ for all $(i,j)\in E(G)$. A minimal vertex cover of $G$ is a vertex cover $S$ of $G$ such that no proper subset of $S$ is a vertex cover of $G$. The minimal generators of $J(G)$ correspond to the minimal vertex covers of $G$, i.e. $J(G)=(\prod\limits_{i\in S}x_i: S\mbox{ is a minimal vertex cover of } G)$.
\begin{Remark}\label{vertexcover}{\em
	Let $e=(i,j)\in E(G)$. Then there exists a minimal vertex cover $S$ of $G$ such that $|\{i,j\}\cap S|= 1$. Take any minimal vertex cover $W$ of $G$. Since $e$ is an edge, we have $|\{i,j\}\cap W|\geq 1$. If $|\{i,j\}\cap W|= 1$, we take $S=W$. Suppose $|\{i,j\}\cap W|= 2$. Then $W\setminus\{i\}$ is not a vertex cover of $G$. Let $(i,k_1),\ldots, (i,k_r)$ be all edges incident to the vertex $i$ such that $\{i,k_s\}\cap W\setminus\{i\}=\emptyset$ for all $1\leq s\leq r$. Then $W'=\{k_1,\ldots,k_r\}\cup W\setminus\{i\}$ is a vertex cover of $G$. Therefore $W'$ contains a minimal vertex cover $S$ of $G$ and $|\{i,j\}\cap S|= 1$.}
\end{Remark}	

A cycle $C$ of length $u$ is a graph with a vertex set $V(C)=\{1,\ldots,u\}$ and the edge set
$E(C)=\{(1,2),\ldots,({u-1},u), (u,1)\}$. A graph $G$ is called a complete graph if each pair of vertices of $G$ are adjacent by an edge. A complete graph with $m$ vertices is denoted by $K_m$. A graph $G$ is called a bipartite graph if the vertex set $V(G)$ can be partitioned into two  sets $A$, $B$ and each edge is of the form $(i,j)\in E(G)$ such that $i\in A$ and $j\in B$. A bipartite graph $G$ with a vertex set $V(G)=A\sqcup B$ is called a complete bipartite graph if every vertex of $A$ is adjacent to every vertex of $B$. We denote a complete bipartite graph by $K_{m,n} $ where the vertex set $V(G)$ is partitioned into two sets $A$ and $B$ with $|A|=m\leq n=|B|$. A graph $G$ is called a complete multipartite graph if the vertex set $V(G)$ can be partitioned into sets $A_1,\ldots,A_k$ for some $k\geq 2$ and $E(G)=\{(a,b): a\in A_i, b\in A_j, 1\leq i,j\leq k, i\neq j \}$.
	\section{Asymptotic behaviour of local $v$-invariant for filtrations}\label{section3}
		In this section, we discuss the asymptotic behaviour of $v(-)$ and $v_\mfp(-)$ for filtrations of homogeneous ideals in an $\mathbb N$-graded Noetherian domain $R$. We mainly consider the following two types of filtrations.
			\begin{enumerate}\label{filtration}
			\item Suppose $I$ is a homogeneous ideal of $R$. We consider the integral closure filtration $\mathcal I=\{I_n=\overline {I^n}\}$ of $I$. By \cite[Theorem 2.7]{Rat2}, $\Ass(R/\overline {I^n})=\Ass(R/\overline{I^{n+1}})$ for all $n\gg 0$. 
			\item\label{asym} Let $I$ and $J$ be  homogeneous ideals in $R$. We consider the generalized symbolic power filtration $\mathcal I=\{I_n=(I^n:J^\infty)\}$ where $I_n:J^\infty=\cup_{t\geq 1}I_n:J^t$ for all $n\geq 1$. By \cite{Br}, we have $\Ass(R/ {I^n})=\Ass(R/{I^{n+1}})$ for all $n\gg 0$. Since $\Ass(R/I_n)=\Ass(R/I^n)\setminus V(J)$ for all $n\geq 1$, we have  $\Ass(R/ {I_n})=\Ass(R/{I_{n+1}})$ for all $n\gg 0$. 
		\end{enumerate}	
	We study the asymptotic behaviour of $v(-)$ and $v_\mfp(-)$ for the above filtrations under the assumption that the filtrations are Noetherian. We mention some cases where it is known that the above-mentioned filtrations are Noetherian.
	\begin{Remark}\label{Noetherian}{\em
		Suppose $R=\bigoplus_{n\in\mathbb N}R_n$ is a Noetherian graded domain and $I$ is a homogeneous ideal.
		\begin{itemize}
			\item[$(1)$] By  \cite[Corollary 5.2.3]{SH}, $\overline{I}$ is a homogenous ideal. If $R$ is an excellent ring then the Rees algebra $R[I]$ is also an  excellent ring. Hence 
			\begin{itemize}
				\item[$(a)$] 	
				its integral closure $S$ in its quotient field is a finite $R[I]$-module  \cite{Mat} . Since $R[I]\subset \bigoplus_{n\in\mathbb N}\overline{I^n}\subset S$, we have $\bigoplus_{n\in\mathbb N}\overline{I^n}$ is a finite $R[I]$ module and
				\item[$(b)$]  there exists an integer $m\geq 1$ such that $\overline{I^n}=\overline{I}^{n-m}\overline{I^m}$ for all $n\geq m$.	 
			\end{itemize}
			In particular, if $R=K[x_1,\ldots,x_n]$ is a polynomial ring over a field $K$ then $\bigoplus_{n\in\mathbb N}\overline{I^n}$ is a finite $R[I]$ module.
			\item[$(2)$]  Let $R=K[x_1,\ldots,x_n]$ be a polynomial ring over a field $K$ and $I,J$ be two monomial ideals. Then the generalized symbolic filtration $\mathcal I=\{I_n=(I^n:J^\infty)\}$ is a Noetherian filtration \cite[Theorem 3.2]{HHT}.
			\end{itemize}
		}
		\end{Remark}
	Now we recall the following result due to Conca \cite{Co} which we will use in our proofs.
	\begin{Lemma}{\em(\cite{Co})}\label{Conca1}
		Let $R=\oplus_{n\in\NN}R_n$ be a graded Noetherian domain and $A=R[y_1,\ldots,y_t]$ be a polynomial ring over $R$. Suppose $\deg y_i=(d_i,1)$ for all $1\leq i\leq t$ and $\deg   x=(n,0)$ for all $x\in R_n$ and for all $n$. Let $U$ be a finitely generated bigraded $A$-module. Then the function $f:\mathbb N\rightarrow \mathbb N$ defined by $f(n)=\min\{w: U_{(w,n)}\neq 0\}$ is eventually a linear function in $n$ with leading coefficient from the set $\{d_1,\ldots,d_t\}$. 
	\end{Lemma}	
	
	Suppose $\mathcal I=\{I_n\}$ is a Noetherian filtration in a Noetherian $\mathbb N$-graded ring $R$ such that $\Ass(R/I_n)=\Ass(R/I_{n+1})$ for all $n\gg 0$. For any $\mfp\in\overline A(\mathcal I)$, we use the following notation:
	$X_\mfp=\{\tilde{\mfp}\in \overline A(\mathcal I): \mfp\subsetneq \tilde{\mfp}\}$ and $Q_\mfp=\prod\limits_{\tilde{p}\in X_\mfp}\tilde{p}$ if $X_\mfp\neq\emptyset$ or $Q_\mfp=R$ if $X_\mfp=\emptyset$.
	
	For any three ideals $I,J$ and $K$ in $R$, by $(I:(J+K^\infty))$, we mean the ideal $(I:J)\cap (I:K^\infty)$.
	\begin{Theorem}\label{integralclosure}
		Let $R$ be an $\mathbb N$-graded Noetherian domain and $I$ be a homogeneous ideal of $R$ generated by homogeneous elements $a_1,\ldots,a_l$ with $\deg a_i=d_i$ for all $1\leq i\leq l$. Consider the filtration $\mathcal I=\{I_n=\overline{I^n}\}$.  Suppose the $R$-algebra $\bigoplus_{n\in\mathbb N}\overline{I^n}$ is a finitely generated $R[I]$-module. Let $\mfp\in\overline A(\mathcal I)$. Then $v_\mfp(\overline{I^n})$ and $v({\overline{I^n}})$  are eventually linear functions in $n$ with leading coefficient from the set $\{d_1,\ldots,d_l\}$. 
	\end{Theorem}
	\begin{proof}
		Let $R[\mathcal I]=\bigoplus_{n\in\mathbb N}\overline{I^n}$. Consider the ideals $J=\bigoplus_{n\in\mathbb N}\overline{I^{n+1}}$, $T=\mfp R[\mathcal I]$ and $Q=Q_\mfp  R[\mathcal I]$  in $R[\mathcal I]$. Thus $(J:_{R[\mathcal I]} T)$ is a finitely generated $R[\mathcal I]$-module and hence $$M=(J:_{R[\mathcal I]} T)/(J:_{R[\mathcal I]} (T+Q^\infty))$$ is a finitely generated $R[\mathcal I]$-module. Therefore $M$ is a finitely generated $R[I]$-module. Note that the $n$th-graded component of $M$ is
		$$(\overline{I^{n+1}}:\mfp)\cap \overline{I^{n}}/(\overline{I^{n+1}}:(\mfp+Q_\mfp^\infty))\cap \overline{I^{n}}.$$ 
		Now $R[I]$ can be represented as a quotient of a polynomial ring $R[y_1,\ldots,y_l]$ bigraded by $\deg y_i=(d_i,1)$ for all $1\leq i\leq l$ and $\deg x=(m,0)$ for all $x\in R_m$. Then $M$ be a finitely generated bigraded $R[y_1,\ldots,y_l]$-module with 
		$$M_{(w,n)}=\Big((\overline{I^{n+1}}:\mfp)\cap \overline{I^{n}}/(\overline{I^{n+1}}:(\mfp+Q_\mfp^\infty))\cap \overline{I^{n}}\Big)_w$$ for all $(w,n)\in\NN^2$. Therefore, by Lemma \ref{Conca1}, the function $f:\mathbb N\rightarrow \mathbb N$ defined by $f(n)=\min\{w: M_{(w,n)}\neq 0\}$ is eventually a linear function in $n$ with leading coefficient from the set $\{d_1,\ldots,d_l\}$. 
		
		Now, by \cite[Corollary 6.8.7]{SH}, we have $$\overline{I^{n+1}}:(\mfp+Q_\mfp^\infty)\subset\overline{I^{n+1}}:\mfp\subset\overline{I^{n+1}}:I=\overline{I^{n}I}:I=\overline{I^{n}}$$ for all $n\geq 1$. 	Hence by \cite[Lemma 1.2]{Co}), we have
		\begin{eqnarray*}
			v_\mfp(\overline{I^{n+1}})&=&\min\{w\in\NN:\big((\overline{I^{n+1}}:\mfp)/(\overline{I^{n+1}}:(\mfp+Q_\mfp^\infty))\big)_w\neq 0\}\\&=&\min\{w\in\NN:M_{(w,n)}\neq 0\}\end{eqnarray*}
		and thus $	v_\mfp(\overline{I^n})$ is eventually a linear function in $n$ with leading coefficient from the set $\{d_1,\ldots,d_l\}$. Since  for all $n\gg 0$, $v(I_n)=\min\{v_\mfp({I_n}):\mfp\in\overline A(\mathcal I)\}$, we have $v({\overline{I^n}})$ is  eventually a linear function in $n$ with leading coefficient from the set $\{d_1,\ldots,d_l\}$. 	
	\end{proof}	
Let $R=\bigoplus_{n\in\mathbb N}R_n$ be a graded Noetherian ring which is finitely generated $R_0$-algebra and $M=\bigoplus_{n\in\mathbb N}M_n$ be a finitely generated $R$-module. For any integer $d\geq 1$ and any integer $0\leq \alpha\leq d-1$, let $R^{(d)}=\bigoplus_{n\in\mathbb N}R_{nd}$ and  $M^{(d,\alpha)}=\bigoplus_{n\in\mathbb N}M_{nd+\alpha}$.
Then $R^{(d)}$ is a graded subring of $R$ and $M^{(d,\alpha)}$ is a graded $R^{(d)}$-module. By \cite[Proposition 3, page 159]{Bo}, there exists an integer $e\geq 1$ such that $R^{(e)}=R_0[R_et^e]$. Let $A=\bigoplus_{n\in\mathbb N}A_n$ where $A_n=R_{ne}$. Then $A$ is a standard graded $A_0$-algebra. For all $0\leq \alpha\leq e-1$, let $N^\alpha=\bigoplus_{n\in\mathbb N}N^\alpha_n$ where $N^\alpha_n=M_{ne+\alpha}$.  

\begin{Lemma}\label{std}
	With the above assumptions and notations, $N^\alpha$ is a finitely generated $A$-module for all $0\leq \alpha\leq e-1$.
\end{Lemma}	 
\begin{proof}
	By \cite[Proposition 2, page 158]{Bo}, $M^{(e,\alpha)}$ is a finitely generated $R^{(e)}$-module for all $0\leq \alpha\leq e-1$. Fix $\alpha\in\{0,\ldots,e-1\}$. Let $\{m_1,\ldots, m_t\}$ be a generating set of $M^{(e,\alpha)}$ as $R^{(e)}$-module and $\deg m_i=n_ie+\alpha$ for all $1\leq i\leq t$. Let $ue+\alpha=\max\{n_1e+\alpha,\ldots,n_te+\alpha\}$. Then for all $n\geq u$, we have
	$M^{\alpha}_{ne+\alpha}=R^{(e)}_{(n-n_1)e}m_1+\cdots+R^{(e)}_{(n-n_t)e}m_t$. Hence for all $n\geq u$ and $m_i\in N_{n_i}^\alpha$ for all $1\leq i\leq t$, we have
	$N^{\alpha}_{n}=A_{n-n_1}m_1+\cdots+A_{n-n_t}m_t$.
\end{proof}	
%
Let $\mathcal I=\{I_n\}$ be a Noetherian filtration in a Noetherian $\mathbb N$-graded ring $R$. Then by \cite[Proposition 3, page 159]{Bo}, there exists an integer $e\geq 1$ such that $I_{en}=I_e^n$ for all $n\geq 1$. The standard Veronese degree of $\mathcal I$ \cite[Definition 3.8]{GS} is defined as
$$ \svd(\mathcal I)=\min\{e: I_{en}=I_e^n\mbox{ for all }n\geq 1\}.$$ 

Note that $\alpha(I_{m+n})\leq \alpha(I_m)+ \alpha(I_n)$ for all $m,n\geq 0$. Thus $\alpha$ is subadditive. By Fekete's Subadditive Lemma, $\lim\limits_{n\to\infty}\frac{\alpha(I_n)}{n}$ exists and $\lim\limits_{n\to\infty}\frac{\alpha(I_n)}{n}=\inf \frac{\alpha(I_n)}{n}$. Therefore $$\displaystyle\lim\limits_{n\to\infty}\frac{\alpha(I_n)}{n}=\lim\limits_{n\to\infty}\frac{\alpha(I_{n\svd(\mathcal I)})}{n\svd(\mathcal I)}=\lim\limits_{n\to\infty}\frac{\alpha(I_{\svd(\mathcal I)}^n)}{n\svd(\mathcal I)}=\frac{\alpha(I_{\svd(\mathcal I)})}{\svd(\mathcal I)}.$$ 
\begin{Theorem}\label{part1}
	Let $R$ be an $\mathbb N$-graded Noetherian domain and $\mathcal I=\{I_n\}$ be a Noetherian filtration of 
	homogeneous ideals of $R$ such that  $\Ass(R/I_n)=\Ass(R/I_{n+1})$ for all $n\gg 0$. Let $\svd(\mathcal I)=:e$. Suppose for all $\mfp\in\overline A(\mathcal I)$,   $I_{n+1}:\mfp\subset I_{n}$ for all $n\gg 0$ (this condition holds for generalized symbolic power filtrations, see Theorem \ref{main}). Then the following hold.
	\begin{itemize}
\item[$(1)$] $v({I_n})$  and $v_\mfp({I_n})$ are quasi-linear type of period $e$ for all $\mfp\in\overline A(\mathcal I)$.
\item[$(2)$]Fix $\mfp\in\overline A(\mathcal I)$. If there exists a homogeneous element $a\in I_1$ of degree $d$ such that the $R_0$-linear map 
$$\Big(\frac{(I_{n+1}:\mfp)}{(I_{n+1}:(\mfp+Q_\mfp^\infty))}\Big)_{w}\xrightarrow{.a}\Big(\frac{(I_{n+2}:\mfp)}{(I_{n+2}:(\mfp+Q_\mfp^\infty))}\Big)_{w+d}$$ is injective for all $w\geq 0$ and $n\gg 0$ then $\displaystyle\lim\limits_{n\to\infty}\frac{v_\mfp(I_n)}{n}$ exists  and $\displaystyle\frac{\alpha(I_e)}{e}\leq \lim\limits_{n\to\infty}\frac{v_\mfp(I_n)}{n}$.

		\end{itemize}
\end{Theorem}
\begin{proof}
$(1)$ 	Let  $\mfp\in\overline A(\mathcal I)$. 
	Since $\mathcal I$ is a Noetherian filtration, the Rees algebra $R[\mathcal I]$ is a finitely generated $R$-algebra. Consider the ideals $J=\bigoplus_{n\in\mathbb N}I_{n+1}$, $T=\mfp R[\mathcal I]$ and $Q=Q_\mfp  R[\mathcal I]$  in $R[\mathcal I]$. Then  $(J:_{R[\mathcal I]} T)$ is a finitely generated $R[\mathcal I]$-module and hence $M=(J:_{R[\mathcal I]} T)/(J:_{R[\mathcal I]} (T+Q^\infty))$ is a finitely generated $R[\mathcal I]$-module. Note that the $m$th-graded component of $M$ is
	$$(I_{m+1}:\mfp)\cap I_{m}/(I_{m+1}:(\mfp+Q_\mfp^\infty))\cap I_{m}.$$
	Now $S=\bigoplus_{n\in\mathbb N}I_{ne}$ is a standard graded $R$-algebra and $$N_\alpha^\mfp=\bigoplus_{n\in\mathbb N}(I_{ne+\alpha+1}:\mfp)\cap I_{ne+\alpha}/(I_{ne+\alpha+1}:(\mfp+Q_\mfp^\infty))\cap I_{ne+\alpha}$$ are finitely generated $S$-modules  for all $\alpha\in\{0,\ldots,e-1\}$. 
	
	Let $I_e=\langle a_1\ldots,a_l\rangle$ with $\deg a_i=d_i$ for all $1\leq i\leq l$. Now $S$ can be represented as a quotient of a polynomial ring $R[y_1,\ldots,y_l]$ bigraded by $\deg y_i=(d_i,1)$ for all $1\leq i\leq l$ and $\deg x=(m,0)$ for all $x\in R_m$.  For all $\alpha\in\{0,\ldots,e-1\}$, the bigrading of 
	$N_{\alpha}^\mfp$ is $$(N_{\alpha}^\mfp)_{(w,n)}=\Big(\frac{(I_{ne+\alpha+1}:\mfp)\cap I_{ne+\alpha}}{(I_{ne+\alpha+1}:(\mfp+Q_\mfp^\infty))\cap I_{ne+\alpha}}\Big)_w$$ for all $(w,n)\in\NN^2$. Then by Lemma \ref{Conca1}, we have $f_\alpha^\mfp (n)=a_\alpha^\mfp n+b_\alpha^\mfp$ for all $n\gg 0$ with $a_\alpha^\mfp\in\{d_1,\ldots,d_l\}$ and $b_\alpha^\mfp\in\ZZ$ for all $\alpha\in\{0,\ldots,e-1\}$ . 

	Since  $I_{n+1}:\mfp\subset I_{n}$ for all $n\gg 0$ and $\mfp\in\overline A(\mathcal I)$, for all $n\gg 0$ and $\alpha\in\{0,\ldots,e-1\}$, we have
$$f_\alpha^\mfp(n):=\min\{w\in\mathbb N: \displaystyle\Big(\frac{(I_{ne+\alpha+1}:\mfp)}{(I_{ne+\alpha+1}:(\mfp+Q_\mfp^\infty))}\Big)_w\neq 0\}=a_\alpha^\mfp n+b_\alpha^\mfp$$ 
with $a_\alpha^\mfp\in\NN_{>0}$, $b_\alpha^\mfp\in\ZZ$. Consider the polynomials $Q_\alpha^\mfp(X)=\displaystyle\frac{a_\alpha^\mfp}{e}X+ (b_\alpha^\mfp-\frac{(\alpha+1) a_\alpha^\mfp}{e})$ for all $0\leq \alpha\leq {e-1} $.

Now by \cite[Lemma 1.2]{Co}, for all $n\gg 0$ with $n\equiv \alpha(mod ~~e)$, we get 
\begin{eqnarray}\label{eq1}
v_\mfp({I_n})= f_{\alpha-1}^\mfp\big(\frac{n-\alpha}{e}\big)=Q_{\alpha-1}^\mfp(n)\mbox{ if }~~~~ \alpha\in\{1,\ldots,e-1\}\mbox{ and}\end{eqnarray}
\begin{eqnarray}\label{eq2}
v_\mfp({I_n})= f_{e-1}^\mfp\big(\frac{n-e}{e}\big)=Q_{e-1}^\mfp(n)\mbox{ if }~~~~ \alpha=0.
\end{eqnarray}

 For all $\alpha\in\{0,\ldots,e-1\}$, consider the polynomials $S_\alpha(X)=c_\alpha X+d_\alpha$ where $c_\alpha=\min\left\{\displaystyle\frac{ a_\alpha^\mfp}{e}:\mfp\in \overline A(\mathcal I)\right\}$ and $d_\alpha=\min\left\{\displaystyle b_\alpha^\mfp-\frac{(\alpha+1) a_\alpha^\mfp}{e}:\mfp\in \overline A(\mathcal I)\mbox{ and }\displaystyle\frac{a_\alpha^\mfp}{e}=c_\alpha\right\}$.

Since  for all $n\gg 0$, $v(I_n)=\min\{v_\mfp({I_n}):\mfp\in\overline A(\mathcal I)\}$, we have  $v(I_n)=S_\alpha(n)$ for all $n\gg 0$ with $n\equiv \alpha(mod ~~e)$ and $\alpha\in\{0,\ldots,e-1\}$.

$(2)$ From the hypothesis of part $2$, we get,
$$f_{e-1}(n)\leq f_{e-2}(n)+d\leq \ldots\leq f_{0}(n)+(e-1)d\mbox{ and } f_{0}(n+1)\leq f_{e-1}(n)+d$$ for all $n\gg 0$.  Hence $\lim\limits_{n\to\infty}\frac{f_i(n)}{n}=\lim\limits_{n\to\infty}\frac{f_j(n)}{n}$ for all $i,j\in \{0,\ldots, e-1\}$. Therefore by equations (\ref{eq1}) and (\ref{eq2}), the limit $\lim\limits_{n\to\infty}\frac{v_\mfp(I_n)}{n}$ exists.	

Let $x\in \mfp$ with $\deg x=\alpha(\mfp)$. For all $n\gg 0$, consider $g_n\in R$ homogeneous such that $\mfp=(I_n:g_n)$ and $v_\mfp(I_n)=\deg g_n$. Then $xg_n\in I_n$ and $\alpha(I_n)\leq \alpha(\mfp)+v_\mfp(I_n)$. Therefore $\alpha(I_n)-\alpha(\mfp)\leq v_\mfp(I_n)$ for all $n\gg 0$ and thus $\displaystyle\frac{\alpha(I_e)}{e}=\lim\limits_{n\to\infty}\frac{\alpha(I_n)}{n}\leq \lim\limits_{n\to\infty}\frac{v_\mfp(I_n)}{n}$.
\end{proof}	
\begin{Example}{\em
	Let $R=K[x,y]$, $\mm=(x,y)R$ where $K$ is a field and $I=(x+y,x^2)R$. Consider the Rees valuation $\om$ corresponding to Rees valuation ring $(V,\mm_{\om})$ which is localization of $R[\frac{x+y}{x^2}]$ at $xR[\frac{x+y}{x^2}]$ (See \cite{SH} for the definition of Rees valuation) and valuation ideals $I(\om)_n:=\{f\in R: \om(f)\geq n\}$ for all $n\in\NN$.  Then by \cite[Example 3.31]{SG}, we have $I(\om)_1=\mm_{\om}\cap R= (x,I)=\mm$, $I(\om)_{2n}=I^n$ and $I(\om)_{2n+1}=\mm I^n$ for all $n\geq 1$. Hence $\bigoplus_{n\in\NN}I(\om)_n$ is a finitely generated $R$-algebra. Note that $\Ass(R/I(\om)_n)=\{\mm\}$ for all $n\geq 1$ and hence $v(I(\om)_n)=v_\mm(I(\om)_n)$ for all $n\geq 1$. 
	\\Let	$z\in (I(\om)_{n+1}:(x))$ for any $n\geq 1$. Then $\om(zx)\geq (n+1)$. Since $\om(x)=1$, we have $\om(z)+1=\om(z)+\om(x)=\om(zx)\geq n+1$ which implies $\om(z)\geq n$. Hence $z\in I(\om)_n$. Therefore for all $n\geq 1$, $(I(\om)_{n+1}:\mm)\subset(I(\om)_{n+1}:(x))\subset I(\om)_n$ . 

Note that  $xy=x(x+y)-x^2\in I$ and  $y^2=y(x+y)-xy\in I$. Since $I^n=\langle{(x^2)^i(x+y)^{n-i}:0\leq i\leq n}\rangle$, we have $(I^n)_w=0$ for all $w< n$.
	
First we show that $(I^n:\mm)_n\neq 0$ and $(I^n:\mm)_s=0$ for all $s<n$. Note that $y(x+y)^{n-1}\in (I^n:\mm)_n\setminus I^n$. Let $s<n$ and $h\in (I^n:\mm)_s$. Then $xh,yh\in (I^n)_{s+1}$. Hence if $s<n-1$ then $h=0$. Suppose $s=n-1$. Then $xh,yh\in (I^n)_{n}$. Therefore there exist $a,b\in R_0$ such that $xh=(x+y)^na$ and $yh=(x+y)^nb$. Hence $xhb=yha$ implies $h=0$. Since $y(x+y)^{n-1}\notin I^n$, 
	we get $\min\{w: \Big(\frac{I^n:\mm}{I^n}\Big)_w\neq 0\}=n$.
	
	Now we show $(\mm I^n:\mm)_n\neq 0$ and $(\mm I^n:\mm)_s=0$ for all $s<n$. Note that $(x+y)^n\in (\mm I^n:\mm)_n\setminus \mm I^n$. Let $s<n$ and $h\in (\mm I^n:\mm)_s$. Then $xh,yh\in (\mm I^n)_{s+1}$. Since $(\mm I^n)_w=0$ for all $w< n+1$, we have $h=0$ for $s<n$.
	
		Consider the filtration $\{I_n=I(\om)_{an}\}$ for any fixed $a\in\NN_{>0}$. Note that for any homogeneous element $u\in \bigoplus_{n\in\NN}I(\om)_{n}$ of degree $d$, $u$ satisfies a polynomial $T^a-u^{a}\in R[\mathcal I][T]$. Hence $\bigoplus_{n\in\NN}I(\om)_n$ is integral over $R[\mathcal I]$. Therefore by Artin-Tate Lemma \cite[Corollary 7.8]{AM}, we have $R[\mathcal I]$ is a finitely generated $R$-algebra. Also note that $I_{n+1}:\mm=I_n$ for all $n\geq 1$ and $x^a\in I_1$ satisfies the hypothesis of Theorem \ref{part1} $(2)$.
	
Let $a$ be an odd integer. Therefore, for all $n\geq 1$, for the filtration $\mathcal I=\{I_n=I(\om)_{an}\}$, 
	we get
	\[\displaystyle v({I_n})=\min\{w: \Big(\frac{I_n:\mm}{I_n}\Big)_w\neq 0\}= \left\{
	\begin{array}{l l}
	~~~~~\displaystyle \min\{w: \Big(\frac{I^{an/2}:\mm}{I^{an/2}}\Big)_w\neq 0\}=\frac{an}{2} \quad \text{if $n$ is even }\\ \vspace{0.3mm}\\
	\displaystyle\min\{w: \Big(\frac{\mm I^{\frac{an-1}{2}}:\mm}{\mm I^{\frac{an-1}{2}}}\Big)_w\neq 0\}=\frac{an-1}{2} \quad \text{if $n$ is odd. }\\
	\end{array} \right.\] 
	Hence $v({I_n})$ is a quasi-linear function in $n$  and $\lim\limits_{n\to\infty}\frac{v(I_n)}{n}=a/2$.
	
Let $a$ be an even integer. Then
	$\displaystyle v({I_n})=\min\{w: \Big(\frac{I^{\frac{an}{2}}:\mm}{I^{\frac{an}{2}}}\Big)_w\neq 0\}=\frac{an}{2}.$ 
	Hence $v({I_n})$ is a linear function in $n$  and $\lim\limits_{n\to\infty}\frac{v(I_n)}{n}=a/2$.}
\end{Example}
\begin{Remark}\label{ideal}{\em
Let $I$ be a homogeneous ideal in a Noetherian $\mathbb N$-graded domain. Then by the same lines of the proof of \cite[Lemma 4.2]{FS}, we get $\lim\limits_{n\to\infty}\frac{v(I^n)}{n}\leq \alpha(I)$. By \cite[Theorem 1.1]{Co}, we have $\lim\limits_{n\to\infty}\frac{v(I^n)}{n}\geq \alpha(I)$. Thus $\lim\limits_{n\to\infty}\frac{v(I^n)}{n}= \alpha(I)$.
}
\end{Remark}	
\begin{Theorem}\label{main}
		Let $R$ be an $\mathbb N$-graded Noetherian domain.	Consider a generalized symbolic power filtration $\mathcal I=\{I_n=(I^n:L^\infty)\}$ where  $I$ and $L$  are two homogeneous ideals of $R$. Suppose $\mathcal I$ is a Noetherian filtration. Then the following hold.
		\begin{enumerate}
			\item[$(1)$] $v_\mfp({I_n})$, $v({I_n})$  are  quasi-linear type for all $\mfp\in\overline A(\mathcal I)$. 
			\item[$(2)$] If one the following holds
			\begin{enumerate}
				\item  $R=K[x_1,\ldots,x_m]$ is a polynomial ring over a field $K$ and $I$ is an equigenerated monomial ideal,
				\item $I$ is an equigenerated homogeneous ideal and $R_0$ is a local ring with infinite residue field,
				\end{enumerate}
				 then $ \displaystyle\lim\limits_{n\to\infty}\frac{v_\mfp(I_n)}{n}$ exist for all $\mfp\in\overline A(\mathcal I)$ and $\displaystyle\lim\limits_{n\to\infty}\frac{v(I_n)}{n}=\frac{\alpha(I_{\svd(\mathcal I)})}{\svd(\mathcal I)}$.
		\item[$(3)$] Let $R=K[x_1,\ldots,x_m]$ be a polynomial ring over a field $K$, $J$ be a cover ideal of a finite simple graph $G$ with $m$ vertices and $\mathcal I=\{J^{(n)}\}$. Then  
		
			$\displaystyle\lim\limits_{n\to\infty}\frac{v(I_n)}{n}, \displaystyle\lim\limits_{n\to\infty}\frac{v_\mfp(I_n)}{n}$ exist for all $\mfp\in\overline A(\mathcal I)$ and $\displaystyle\lim\limits_{n\to\infty}\frac{v(I_n)}{n}=\frac{\alpha(J^{(2)})}{2}$.
			
		\end{enumerate}	
\end{Theorem}	
\begin{proof}
$(1)$ We have $\Ass(R/I_n)=\Ass(R/I_{n+1})$ for all $n\gg 0$ (see (\ref{asym}), the first paragraph of section \ref{section3}). Fix $\mfp\in\overline A(\mathcal I)$. We show that for all $n\gg 0$, we have $(I_{n+1}:\mfp)\subset I_n$. By \cite[Theorem 4.1]{Rat}, there exists an integer $k\geq 1$ such that $I^{n+1}:I=I^n$ for all $n\geq k$. Fix any $n\geq k$. Let $t\geq 1$ be an integer such that $I^n:L^\infty=I^n:L^t$ and $I^{n+1}:L^\infty=I^{n+1}:L^t$. Let $x\in (I_{n+1}:\mfp)$. Then $xI\subset x\mfp\subset I_{n+1}= I^{n+1}:L^t$. Thus $xL^tI\subset I^{n+1}$ and hence $xL^t\subset I^{n+1}:I=I^n$. Therefore $x\in I^n:L^t=I^n:L^\infty=I_n$. Since $\mathcal I$ is a Noetherian filtration, by Theorem \ref{part1}, we get, $v_\mfp({I_n})$ and $v({I_n})$  are  quasi-linear type.

$(2)$ Let $\mfp\in\overline A(\mathcal I)$. Suppose $I$ is an equigenerated monomial ideal in $R=K[x_1,\ldots,x_m]$  and $I$ is generated by monomials $a_1,\ldots,a_l$ of degree $d$. Let $a=a_1+\cdots+a_l$. For all $n\gg 0$, we have $$(I^{n+1}:(a))=\bigcap\limits_{i=1}^l(I^{n+1}:(a_i))=(I^{n+1}:I)= I^n$$ where the last equality follows from  \cite[Theorem 4.1]{Rat}.  

Suppose $I$ is an equigenerated homogeneous ideal, generated by homogeneous elements of degree $d$ and $R_0$ is a local ring with infinite residue field.  Then there exists a homogeneous $I$-superficial element $a\in I$ of degree $d$, i.e., there exists a homogeneous element $a\in I$ of degree $d$ and a positive integer $c$ such that $(I^{n+1}:(a))\cap I^c=I^n$ for all $n\geq c$. Since $R$ is domain, using Artin-Rees Lemma,  we get $(I^{n+1}:(a))= I^n$  for all $n\gg 0$.

Now we show that for all $\mfp\in\overline A(\mathcal I)$ and $n\gg 0$, we have
 $$\displaystyle((I_{n+2}:(\mfp+Q_\mfp^\infty)):(a))\bigcap (I_{n+1}:\mfp)=(I_{n+1}:(\mfp+Q_\mfp^\infty)).$$ 
Let $r, m\geq 1$ be integers such that $I^n:L^\infty=I^n:L^r$, $I^{n+1}:L^\infty=I^{n+1}:L^r$, $I_{n+1}:Q_\mfp^\infty=I_{n+1}:Q_\mfp^m$ and $I_{n+2}:Q_\mfp^\infty=I_{n+2}:Q_\mfp^m$. Let $y\in (I_{n+1}:\mfp)$ be a homogeneous element such that $ya\in (I_{n+2}:(\mfp+Q_\mfp^\infty))$. Then 
  $ya (\mfp+Q_\mfp^m) L^r\subset I^{n+2}$ and hence $y(\mfp+Q_\mfp^m)L^r\subset (I^{n+2}:(a))=I^{n+1}.$ Therefore $y(\mfp+Q_\mfp^m)\subset I_{n+1}$. Hence by Theorem \ref{part1}, $ \displaystyle\lim\limits_{n\to\infty}\frac{v_\mfp(I_n)}{n}$ exists.

 Let $\svd(\mathcal I)=e$. Note that for all $k\gg 0$,
$\Ass(R/I_k)=\Ass(R/I_{ek})=\Ass(R/I_e^k)$. By \cite{Co}, there exists a prime $\mfq\in \Ass(R/I_e^k)$ for all $k\gg 0$ such that $v(I_e^n)=v_\mfq(I_e^n)$ for all $n\gg 0$. By Remark \ref{ideal}, for all $n\gg 0$, we have
$$\lim\limits_{n\to\infty}\frac{v_\mfq(I_{n})}{n}=\lim\limits_{n\to\infty}\frac{v_\mfq(I_{en})}{en}=\lim\limits_{n\to\infty}\frac{v_\mfq(I_e^n)}{en}=\lim\limits_{n\to\infty}\frac{v(I_e^n)}{en}=\frac{\alpha(I_e)}{e}.$$

Since, by Theorem \ref{part1}, $\displaystyle\frac{\alpha(I_e)}{e}\leq \lim\limits_{n\to\infty}\frac{v_\mfp(I_n)}{n}$ for all $\mfp\in\overline A(\mathcal I)$, we have $$\displaystyle\lim\limits_{n\to\infty}\frac{v(I_n)}{n}=\lim\limits_{n\to\infty}\frac{v_\mfq(I_{n})}{n}=\frac{\alpha(I_{e})}{e}.$$

$(3)$ For any edge $e=(i,j)$ in $G$, we denote $P_e=(x_i,x_j)$. Let $e_1,\ldots,e_l$ be edges of $G$. We denote $P_{e_i}$ by $P_i$ for all $1
\leq i\leq l$. Then $J=\bigcap\limits_{i=1}^lP_i$ and by \cite[Theorem 3.7]{CEHH},  $J^{(n)}=\bigcap\limits_{i=1}^lP_i^n$ for all $n\geq 1$.  Fix $i\in\{1,\ldots,l\}$ and let $P_i=(x_{i_1},x_{i_2})$ for some $x_{i_1},x_{i_2}\in\{x_1.\ldots,x_m\}$.

By Remark \ref{vertexcover}, there exists a minimal vertex cover $S$ of $G$ such that $|\{i_1,i_2\}\cap S|=1$. Since $(i_1,i_2)\in E(G)$, without loss of generality we assume, ${i_1}\in S$ and $i_2\notin S$. Let $a\in J$ be the monomial generator of $J$ that corresponds to $S$. We show that $$((J^{(n+2)}):(a))\bigcap (J^{(n+1)}:P_i)=J^{(n+1)}\mbox{ for all }n\gg 0.$$
Let $y\in (J^{(n+1)}:P_i)$ be a homogeneous element and $ya\in J^{(n+2)}$.  Since $y\in (J^{(n+1)}:P_i)$, we have $y\in P_j^{n+1}$ for all $j\in\{1,\ldots,l\}\setminus\{i\}$.

Note that $ya\in J^{(n+2)}$ implies $y\in (P_i^{n+2}:(a))=\big((P_i^{n+2}:(x_{i_1})):(\frac{a}{x_{i_1}})\big).$  
Now $(P_i^{n+2}:(x_{i_1}))=P_i^{n+1}$. Since $x_{i_1}\nmid \frac{a}{x_{i_1}}$ and $x_{i_2}\nmid \frac{a}{x_{i_1}}$, we have $(P_i^{n+1}:\frac{a}{x_{i_1}})=P_i^{n+1}$. Therefore $y\in \bigcap\limits_{i=1}^lP_i^{n+1}=J^{(n+1)}$. Hence by Theorem \ref{part1}, $\displaystyle\lim\limits_{n\to\infty}\frac{v_{P_i}(I_n)}{n}$ exists.

By \cite[Theorem 5.1]{HHT}, $\svd(\mathcal I)=2$. Thus using the same lines of the last paragraph of the proof of part $2$, we get the required result. 
\end{proof}	
\begin{Remark}{\em
Recently in \cite{MRK}, the authors proved that for any Noetherian filtration $\{I_n\}$, the limit $\displaystyle\lim\limits_{n\to\infty}\frac{v(I_n)}{n}$ exists and for any cover ideal $J$, $\displaystyle\lim\limits_{n\to\infty}\frac{v(J^{(n)})}{n}=\frac{\alpha(J^{(2)})}{2}$. }
\end{Remark}	
\begin{Example}{\em
Consider the homogeneous ideal 	$I=(xy,xz,xw,yz)$ in the polynomial ring $R=K[x,y,z,w]$ where $K$ is a field. Note that $I=(x,y)\cap(x,z)\cap(y,z,w)$ and $P_1=(x,y)$, $P_2=(x,z)$ and $P_3=(y,z,w)$ are associated primes of $I$. 	By \cite[Theorem 3.7]{CEHH}, for all $n\geq 1$, we have $I^{(n)}=\bigcap\limits_{i=1}^3 P_{i}^n$. By Remark \ref{Noetherian}, $\{I^{(n)}\}$ is a Noetherian filtration. 

For all $n\geq 1$, consider $\displaystyle L^{1} _n=(\frac{I^{(n)}:P_1}{I^{(n)}})=\frac{P_1^{n-1}\cap P_2^n\cap P_3^n}{I^{(n)}}$. Let $x^{\alpha_1}y^{\alpha_2}z^{\alpha_3}w^{\alpha_4}+I^{(n)}$ be a nonzero element in $L_n^{1}$. Then $\alpha_1+\alpha_2=n-1$, $\alpha_1+\alpha_3\geq n$ and $\alpha_2+\alpha_3+\alpha_4\geq n$.

{\underline{Case 1:}} Let $n=2k$ for some positive integer $k$. Note that $x^ky^{k-1}z^{k+1}+I^{(n)}\neq0$ in $(L^{1} _n)_{3k}.$ We show $(L^{1} _n)_{s}=0$ for all $s<3k$. Suppose $\alpha_1+\alpha_2+\alpha_3+\alpha_4<3k$. Then $ \alpha_3+\alpha_4<3k-2k+1=k+1$ implies $2k\leq \alpha_2+\alpha_3+\alpha_4\leq\alpha_2+k$. Therefore $\alpha_2\geq k$ and hence $\alpha_1\leq 2k-1+k=k-1$. Then  $2k\leq\alpha_1+\alpha_3\leq k-1+\alpha_3$  implies $k+1\leq\alpha_3$ which contradicts that $ \alpha_3+\alpha_4<k+1$.

{\underline{Case 2:}} Let $n=2k+1$ for some nonnegative integer $k$. Note that $x^ky^kz^{k+1}+I^{(n)}\neq0$ in $(L^{1} _n)_{3k+1}$. We show $(L^{1} _n)_{s}=0$ for all $s<3k+1$. Suppose $\alpha_1+\alpha_2+\alpha_3+\alpha_4<3k+1$. Then $ \alpha_3+\alpha_4<3k+1-2k=k+1$. Hence $2k+1\leq \alpha_2+\alpha_3+\alpha_4\leq\alpha_2+k.$ Therefore $\alpha_2\geq k+1$. Thus $\alpha_1=2k-\alpha_2\leq2k-k-1=k-1$. Then $2k+1\leq\alpha_1+\alpha_3\leq k-1+\alpha_3$ implies $k+2\leq\alpha_3$ which contradicts  $\alpha_3+\alpha_4\leq k$.

 Thus  by  \cite[Lemma 1.2]{Co}, we get $$v_{P_1}(I^{(n)})=\begin{cases}
\frac{3}{2}n & \mbox{ if n is even}\\ \frac{3}{2}n-\frac{1}{2}& \mbox{ if n is odd}.
\end{cases}$$ Similar calculation shows $v_{P_1}(I^{(n)})=v_{P_2}(I^{(n)})=v_{P_3}(I^{(n)})$.
} 
\end{Example}		
\section{$v$-numbers of symbolic powers of cover ideals of graphs}\label{coverideal}
In this section, we compute  $v$-numbers of cover ideals of finite simple graphs and relate these $v$-numbers to the Castelnuovo-Mumford regularity of those ideals. Throughout this section, for a graph $G$, we use the notation $R$ to denote the polynomial ring with  $|V(G)|$ variables over a field $K$.	
\begin{Theorem}\label{combi}
	Let $G=K_{p_1,p_2}$ be a complete bipartite graph with the vertex set $V=V_1\sqcup V_2$ with $|V_i|=p_i$ and $p_1\leq p_2$. Let $J$ be the cover ideal of $G$. Then 
	\begin{enumerate}
		\item[$(1)$] $v(J^{n})=np_1+p_2-2$ for all $n\geq 1$. 
		\item[$(2)$] $\reg(R/J^{n})-v(J^{n})= n(p_2-p_1)-(p_2-p_1)$ for all $n\geq 1$ and $v(J)= \reg(R/J)$.
		\item[$(3)$] If $p_2=p_1$  then $\reg(R/J^{n})=v(J^{n})$ for all $n\geq 1$.
		\item[$(4)$] If $p_2=p_1+1$  then $\reg(R/J^{n})-v(J^{n})=n-1$ for all $n\geq 1$.	
	\end{enumerate}	
\end{Theorem}	
\begin{proof}
	By \cite[Corollary 2.6]{GRV1}, $J^{(n)}=J^n$ for all $n\geq 1$. Let $x_{1},\ldots,x_{p_1}$ correspond the vertices in $V_1$ and $y_{1},\ldots,y_{p_2}$ correspond the vertices in $V_2$. Let  $P_{ij}=(x_{i},y_{j})$ where $1\leq i\leq p_1,1\leq j\leq p_2$.	Then $J=\bigcap\limits_{\substack{1\leq i\leq p_1\\1\leq j\leq p_2}} P_{ij}$.
	By \cite[Theorem 3.7]{CEHH}, for all $n\geq 1$, we have $J^{(n)}=\bigcap\limits_{\substack{1\leq i\leq p_1\\1\leq j\leq p_2}} P_{ij}^n$. By Remark \ref{Noetherian}, $\{J^{(n)}\}$ is a Noetherian filtration. Fix $i\in\{1,\ldots,p_1\}$ and $j\in\{1,\ldots,p_2\}$. For any $n\geq 1$, let $$L_n^{ij}:=\displaystyle \frac{J^{(n)}:P_{ij}}{J^{(n)}}=\frac{P_{ij}^{n-1}\cap(\bigcap\limits_{\substack{\substack{1\leq l\leq p_1\\1\leq k\leq p_2}\\(l,k)\neq (i,j)}} P_{lk}^n)}{\bigcap\limits_{\substack{1\leq l\leq p_1\\1\leq k\leq p_2}}P_{lk}^n}$$ and $a=np_1+p_2-2$. Note that $\displaystyle x_{i}^{n-1}\left(\prod\limits_{\substack{1\leq r\leq p_1\\ r\neq i}}x_{r}^n\right)\left(\prod\limits_{\substack{1\leq s\leq p_2\\ s\neq j}}y_{s}\right)+J^{(n)}\in (L_n^{ij})_a.$ Let $\prod\limits_{r=1}^{p_1}x_r^{\alpha_r}\prod\limits_{s=1}^{p_2}y_s^{\beta_s}+J^{(n)}$ be a nonzero element in $L_n^{ij}$.
	Then
	\begin{enumerate}
		\item $\alpha_{i}+\beta_{j}=n-1$,
		\item $\alpha_{i}+\beta_{s}\geq n$ for all $s\in\{1,\ldots,p_2\}\setminus\{j\}$,
		\item $\alpha_{r}+\beta_{j}\geq n$ for all $r\in\{1,\ldots,p_1\}\setminus\{i\}$.
	\end{enumerate}	
	Let $\alpha_{i}=e$ for some $0\leq e\leq n-1$. Then $\beta_j=n-e-1$. Thus 
	$$n-1+(e+1)(p_1-1)+(n-e)(p_2-1)\leq n-1+\sum\limits_{\substack{1\leq r\leq p_1\\r\neq i}}\alpha_r+\sum\limits_{\substack{1\leq s\leq p_2\\s\neq j}}\beta_s\leq\sum\limits_{r=1}^{p_1}\alpha_r+\sum\limits_{s=1}^{p_2}\beta_s.$$
	Let $p_2=p_1+l$ for some integer $l\geq 0$. Then
	\begin{eqnarray*}
		n-1+(e+1)(p_1-1)+(n-e)(p_2-1)&=& n-1+ep_1-e+p_1-1+np_2-ep_2-n+e\\&=&	ep_1+p_1+np_2-ep_1-el-2\\&=& p_1+np_2-el-2\\&=& p_1+np_1+nl-el-2\\&=&np_1+p_1+(n-e)l-2\\&\geq& np_1+p_1+l-2=np_1+p_2-2=a.
	\end{eqnarray*}	
	Thus $(L_n^{ij})_w=0$ for all $w<a$ and  by  \cite[Lemma 1.2]{Co}, we get the required result.
	
	$(2)$	By \cite[Lemma 3 and Theorem 4]{KKSS}, for all $n\geq 1$, we have $\reg(R/J^{(n)})=np_2+p_1-2$. Hence, for all $n\geq 1$, we get 
	$$	\reg(R/J^{(n)})-v(J^{(n)})= np_2+p_1-2-np_1-p_2+2=n(p_2-p_1)-(p_2-p_1).$$	
	$(3)$ and $(4)$  follow from $(2)$.	
\end{proof}
As a consequence of Theorem \ref{combi}, we answer the question \cite[Question 3.12]{KSa}, posed by Saha, using the following construction due to Seyed Fakhari \cite[Section 3]{FS2018}.
\begin{Remark}\label{newgraph}{\em
		Let $G$ be a graph with vertex set $V(G)=\{x_1,\ldots,x_n\}$. In \cite[Section 3]{FS2018}, for every integer $k\geq 1$, Seyed Fakhari constructed a new graph $G_k$ with the vertex set $V(G_k)=\{x_{i,p}: 1\leq i\leq n, 1\leq p\leq k\}$ and the edge set $E(G_k)=\{(x_{i,p},x_{j,q}): (x_i,x_j)\in E(G), p+q\leq k+1\}$. Let $I^{\pol}$ denote the  polarization of a monomial ideal  $I$ (see \cite[Section 1.6]{HH} for details of polarization). Seyed Fakhari  also showed that the polarization of $k$-th symbolic power of the cover ideal $J(G)$ of $G$ is the cover ideal $J(G_k)$ of $G_k$, i.e., $(J(G)^{(k)})^{\pol}=J(G_k)$ for all $k\geq 1$ \cite[Lemma 3.4]{FS2018}. By \cite[Corollary 3.5]{SS}, \cite[Theorem 4.1]{F23}, we have $v(J(G_k))=v((J(G)^{(k)})^{\pol})=v(J(G)^{(k)})$ for all $k\geq 1$. 
	}
\end{Remark}
\begin{Corollary}\label{answer}
	For every positive integer $p\geq 2$, there exists an unmixed bipartite graph $\mathcal H_p$ which is not a complete multipartite graph and $v(J(\mathcal H_p))-\bight(I(\mathcal H_p))=p-2$ where $\bight(I(\mathcal H_p))=\max\{\height(Q):Q\in\Ass(I(\mathcal H_p))\}$. In particular, $v(J(\mathcal H_p))\geq\bight(I(\mathcal H_p))$.
\end{Corollary}	
\begin{proof}
For any positive integer $p\geq 2$, let $G=K_{p,p}$ be a complete bipartite graph with $V(G)=V_1\sqcup V_2$, $|V_1|=p=|V_2|$. Let $V_1=\{a_i: 1\leq i\leq p\}$ and $V_2=\{b_j:1\leq j\leq p\}$. Construct  the graph $G_2$ from $G$ as in Remark \ref{newgraph}.  Then the vertex set of $G_2$ is $V(G_2)=\{a_{i,1},a_{i,2},b_{j,1},b_{j,2} : 1\leq i,j\leq p\}$ and  $E(G_2)=\{(a_{i,1},b_{j,1}),(a_{i,1},b_{j,2}),(a_{i,2},b_{j,1}):1\leq i,j\leq p\}$.
	
		Consider the graph $\mathcal H_p$ with $V(\mathcal H_p)=A\sqcup B$ with $A=\{x_1,\ldots, x_{2p}\}$ and $B=\{y_1,\ldots,y_{2p}\}$ where $x_i=a_{i,2}$, $x_{p+r}=a_{r,1}$  for all $1\leq i, r\leq p$ and $y_j=b_{j,1}$, $y_{p+s}=b_{s,2}$  for all $1\leq j,s\leq p$ and $E(\mathcal H_p)=\{(x_i,y_j), (x_{p+i},y_s): 1\leq i, j\leq p, 1\leq s\leq 2p\}$. Then $\mathcal H_p$ is a bipartite graph and by \cite[Theorem 1.1]{Vil}, $\mathcal H_p$ is unmixed. Thus $\bight(I(\mathcal H_p))=2p$.
		
		Since $(x_1,y_{p+1})\notin E(\mathcal H_p)$, $\mathcal H_p$ is not a complete bipartite graph. Note that the graph $\mathcal H_p$ is obtained from $G_2$ by renaming $G_2$. Thus By Theorem \ref{combi} and Remark \ref{newgraph}, we have
		$v(J(\mathcal H_p))=v(J(G_2))=v(J(G)^{(2)})^{\pol})=v((J(G)^{(2)})=3p-2$. Hence $v(J(\mathcal H_p))-\bight(I(\mathcal H_p))=3p-2-2p=p-2$.
		
		Note that there does not exist any complete multipartite graph $H$ such that $V(H)=V(\mathcal H_p)$ and $E(H)=E(\mathcal H_p)$. Suppose there exists a multipartite graph $H$ with $V(H)=A_1\sqcup\cdots\sqcup A_t$ such that $V(H)=V(\mathcal H_p)$ and $E(H)=E(\mathcal H_p)$. 
		
		{\underline{Case 1:}} Suppose $x_1,y_{p+1}\in A_i$ for some $i\in\{1,\ldots,t\}$ . Then $y_1\in A_l$ for some $l\in\{1,\ldots,t\}\setminus\{i\}$ as $(x_1,y_1)\in E(H)$. Since $(y_1,y_{p+1})\notin E(H)$, we have that $H$ is not a complete multipartite graph.
		
		{\underline{Case 2:}} Suppose $x_1\in A_i$ and $y_{p+1}\in A_j$ for some $i,j\in\{1,\ldots,t\}$ with $i\neq j$. Since $(x_1,y_{p+1})\notin E(H)$, we have that $H$ is not a complete multipartite graph.	
\end{proof}	
We use the following lemma in our next theorem.
\begin{Lemma}\label{tech}
	Let $a,b,c$ be nonnegative integers and $n$ be a positive integer such that $a+b=n-1$, $a+c\geq n$ and $b+c\geq n$ then 
	$$c\geq\begin{cases}
	\displaystyle \frac{n}{2}+1& \mbox{ if n is even}\\ \displaystyle\frac{n+1}{2}& \mbox{ if n is odd}.
	\end{cases}$$
\end{Lemma}	
\begin{proof}
	Let $a=t$ for some $0\leq t\leq n-1$. Then $b=n-1-t$. Thus $t+c\geq n$ and $-t+c\geq 1$. Therefore $c\geq \lceil{\frac{n+1}{2}}\rceil$. 
\end{proof}	

\begin{Theorem}\label{complete}
	Let $K_m$ be a complete graph with $m$ vertices with $m\geq 3$, $J$ be the cover ideal of $K_m$ and  $\mathcal I=\{J^{(n)}\}$. Then
	\begin{enumerate}	
		\item[$(1)$]  for any $\mfp\in\overline A(\mathcal I)$,  for all $n\geq 1$, \[\displaystyle v({J^{(n)}})= \left\{
		\begin{array}{l l}
		~~~~~\displaystyle \frac{m}{2}n+m-3& \quad \text{if $n$ is even }\\ \vspace{0.3mm}\\
		\displaystyle\frac{m}{2}n+\frac{m}{2}-2& \quad \text{if $n$ is odd. }\\ 
		\end{array} \right.\] 
		\item[$(2)$] $v(J^{(n)})\leq\reg(R/J^{(n)})=n(m-1)-1$ for all $n\geq 1$. 
		\item[$(3)$] $\reg(R/J^{(i)})=v(J^{(i)})$ for $i=1,2$. 
	\end{enumerate}			
	\begin{proof}
		Let $P_{ij}=(x_i,x_j)$ with $1\leq i<j\leq m$. Note that $J=\bigcap\limits_{{1\leq i<j\leq m}}P_{ij}$ and by \cite[Theorem 3.7]{CEHH}, for all $n\geq 1$, we have $J^{(n)}=\bigcap\limits_{\substack{1\leq i<j\leq m}}P_{ij}^n$. By Remark \ref{Noetherian}, $\{J^{(n)}\}$ is a Noetherian filtration. Fix  a pair $(i,j)$ such that $1\leq i<j\leq m$. For all $n\geq 1$,  we have,
		$$L_n^{ij}:=\displaystyle \frac{J^{(n)}:P_{ij}}{J^{(n)}}=\frac{P_{ij}^{n-1}\cap(\bigcap\limits_{\substack{1\leq l<k\leq m\\(l,k)\neq (i,j)}} P_{lk}^n)}{\bigcap\limits_{\substack{1\leq l<k\leq m}}P_{lk}^n}.$$
		
		Let $\prod\limits_{i=1}^mx_i^{\alpha_i}+J^{(n)}$ be a nonzero element in $L_n^{ij}$. Then
		\begin{enumerate}
			\item $\alpha_i+\alpha_j=n-1$,
			\item $\alpha_i+\alpha_k\geq n$ for all $k\in\{1,\ldots,m\}\setminus \{j\}$,
			\item	$\alpha_j+\alpha_k\geq n$ for all $k\in\{1,\ldots,m\}\setminus \{i\}$.
		\end{enumerate}		
		Suppose $n$ is an even positive integer. Then 
		by Lemma \ref{tech}, $\alpha_k\geq \frac{n}{2}+1$ for all $k\in\{1,\ldots,n\}\setminus \{i,j\}$. Then $$\sum\limits_{i=1}^n\alpha _i\geq n-1+(m-2)(\frac{n}{2}+1)=\frac{m}{2}n+m-3.$$ 
		Note that $x_i^\frac{n}{2}x_j^{(\frac{n}{2}-1)}\prod\limits_{\substack{k=1\\k\neq i,j}}^mx_k^{\frac{n}{2}+1}+J^{(n)}\in L_n^{ij}$.	
		
		Suppose $n$ is an odd positive integer. Then 
		by Lemma \ref{tech}, $\alpha_k\geq \frac{n+1}{2}$ for all $k\in\{1,\ldots,n\}\setminus \{i,j\}$. Then $$\sum\limits_{i=1}^n\alpha _i\geq n-1+(m-2)(\frac{n+1}{2})=\frac{m}{2}n+\frac{m}{2}-2.$$ 
		Note that $x_i^\frac{n-1}{2}x_j^{\frac{n-1}{2}}\prod\limits_{\substack{k=1\\k\neq i,j}}^mx_k^{\frac{n+1}{2}}+J^{(n)}\in L_n^{ij}$.	Therefore
		
		\[\displaystyle v_{P_{i,j}}({J^{(n)}})= \left\{
		\begin{array}{l l}
		~~~~~\displaystyle \frac{m}{2}n+m-3& \quad \text{if $n$ is even }\\ \vspace{0.3mm}\\
		\displaystyle\frac{m}{2}n+\frac{m}{2}-2& \quad \text{if $n$ is odd. }\\ 
		\end{array} \right.\] 
		Hence,  by  \cite[Lemma 1.2]{Co}, we get the required expression for $v(J^{(n)})$.
		
		$(2)$ Since complete graphs are claw-free, by  \cite[Corollary 3.8]{Fa}, we have $\reg(R/J^{(n)})=n(m-1)-1$ for all $n\geq 1$. Thus by part $(1)$, we get the required result.
		
		$(3)$  Again using  \cite[Corollary 3.8]{Fa} and the first part of the proof, we have $\reg(R/J^{(1)})=m-2=v(J)$ and $\reg(R/J^{(2)})=2m-3=v(J^{(2)})$. 
	\end{proof}		
\end{Theorem}
The next result shows that for both connected bipartite graphs and  connected non-bipartite graphs, the difference between the regularity and the $v$-number of the cover ideal can be arbitrarily large. This strengthens and gives an alternative proof of \cite[Theorem 3.10]{KSa} (note that the example in \cite[Theorem 3.10]{KSa} is a non-bipartite graph). 
\begin{Theorem}\label{arbitrary}
	 For every positive integer $k$, 
\begin{enumerate}
\item[$(1)$] there exist infinitely many connected bipartite graphs $\mathcal H_k$ with the cover ideal $J(\mathcal H_k)$ such that $\reg(S/J(\mathcal H_k))-v(J(\mathcal H_k))=k$ where $S$ is the polynomial ring over a field $K$ with $|V(\mathcal H_k)|$ variables.
\item[$(2)$] there exists a connected non-bipartite graph $\mathcal H_k$ with the cover ideal $J(\mathcal H_k)$ such that $\reg(S/J(\mathcal H_k))-v(J(\mathcal H_k))=k$ where $S$ is the polynomial ring over a field $K$ with $|V(\mathcal H_k)|$ variables.
	\end{enumerate}
\end{Theorem}
\begin{proof}
	$(1)$ Let $G=K_{m,m+1}$ be a complete bipartite graph with $V(G)=A\sqcup B$, $|A|=m$ and $|B|=m+1$ for any positive integer $m$. Let $A=\{x_1,\ldots,x_m\}$ and $B=\{y_1,\ldots,y_{m+1}\}$.
	For any integer $k\geq 1$, we show that $G_k$ is  a connected bipartite graph where $G_k$ is the new graph constructed from $G$ as mentioned in Remark \ref{newgraph}. Note that $V(G_k)=A_k\sqcup B_k$ with $A_k=\{x_{i,p}: 1\leq i\leq m, 1\leq p
	\leq k\}$ and $B_k=\{y_{j,q}: 1\leq j\leq m+1, 1\leq q
	\leq k\}$. Consider two distinct vertices $x_{i,p},x_{l,t}\in A_k$ of $G_k$ for some $1\leq i,l\leq m$ and $1\leq p,t\leq k$. Then we have a path $(x_{i,p},y_{1,1}), (x_{l,t},y_{1,1})$ in $G_k$. Similarly for any two distinct vertices $y_{j,q},y_{r,s}\in B_k$ with $1\leq j,r\leq m+1$ and $1\leq q,s\leq k$,  we have a path $(x_{1,1},y_{j,q}), (x_{1,1},y_{r,s})$ in $G_k$. Let $x_{i,p}\in A_k$ and $y_{j,q}\in B_k$ be two vertices in $G_k$ where $1\leq i\leq m$, $1\leq j\leq m+1$ and $1\leq p,q\leq k$. Then we have a path $(x_{i,p},y_{1,1}),(x_{1,1},y_{1,1}),(x_{1,1},y_{j,q})$ in $G_k$.
	
	For any integer $k\geq 1$, consider the graph $\mathcal H_k=G_{k+1}$. Then using  Theorem \ref{combi} $(4)$, \cite[Corollary 1.6.3]{HH}, Remark \ref{newgraph} and \cite[Corollary 2.6]{GRV1}, we have
	\begin{eqnarray*}
		\reg(S/J(\mathcal H_k))-v(J(\mathcal H_k))	&=&\reg(J(\mathcal H_k))-v(J(\mathcal H_k))-1
		\\&=&\reg(J(G_{k+1}))-v(J(G_{k+1}))-1\\&=& \reg((J(G)^{({k+1})})^{\pol})-v((J(G)^{({k+1})})^{\pol})-1\\&=& \reg(J(G)^{({k+1})})-v(J(G)^{({k+1})})-1\\&=&\reg(R/J(G)^{({k+1})})-v(J(G)^{({k+1})})\\&=&\reg(R/J(G)^{{k+1}})-v(J(G)^{{k+1}})=k+1-1=k
	\end{eqnarray*}	where $R$ is the polynomial ring over the field $K$ with $|V(G)|$-many variables.
Note that for a fixed positive integer $k$ and for every positive integer $m$, we consider $G=K_{m,m+1}$ and construct $\mathcal H_k$. Hence for a fixed positive integer $k$, we get infinitely many graphs $\mathcal H_k$. 

$(2)$ Let $G=K_3$ be the complete graph with the vertex set $V(G)=\{x_1,x_2,x_3\}$. Then for every positive integer $k$, $G_k$ is the graph with the vertex set $V(G)=\{x_{i,p}: 1\leq i\leq 3, 1\leq p\leq k\}$ constructed from $G$ as mentioned in Remark \ref{newgraph}. Note that for any two distinct vertices $x_{i,p}$ and $x_{j,q}$ with $1\leq i,j\leq 3$, $1\leq p,q\leq k$, we have path $(x_{i,p},x_{t,1}),(x_{t,1},x_{j,q})$ where $t\in\{1,2,3\}\setminus\{i,j\}$. Hence $G_k$ is connected for all $k\geq 1$.  For all $k\geq 1$, the graph $G_k$ has a $3$-cycle with vertices $x_{1,1}, x_{2,1}, x_{3,1}$ and hence $G_k$ is a non-bipartite graph.

For any integer $k\geq 1$, consider the graph $\mathcal H_k=G_{2k+1}$. Then using Theorem \ref{complete}, \cite[Corollary 1.6.3]{HH} and \cite[Corollary 3.8]{FS},  we have
\begin{eqnarray*}
	\reg(S/J(\mathcal H_k))-v(J(\mathcal H_k))	&=&\reg(J(\mathcal H_k))-v(J(\mathcal H_k))-1
	\\&=&\reg(J(G_{2k+1}))-v(J(G_{2k+1}))-1\\&=& \reg((J(G)^{({2k+1})})^{\pol})-v((J(G)^{({2k+1})})^{\pol})-1\\&=& \reg(J(G)^{({2k+1})})-v(J(G)^{({2k+1})})-1\\&=&\reg(R/J(G)^{({2k+1})})-v(J(G)^{({2k+1})})\\&=&2.(2k+1)-1-(\frac{3}{2}(2k+1)-\frac{1}{2})=k
\end{eqnarray*}	where $R$ is the polynomial ring over the field $K$ with $|V(G)|$-many variables.
\end{proof}	
Recently in \cite{MRK}, it is shown that for a cover ideal $J(G)$ of a graph $G$, $v(J(G)^{(n)})\leq \reg(R/J(G)^{(n)})$ for all $n\geq 1$. Motivated by Theorem \ref{combi} and Theorem \ref{complete}, we pose the following question.
\begin{Question}\label{question}	
	For every  integer $m>2$, does there exist a graph $G_m$ with the cover ideal $J_m$ such that $\reg(R/J_m^{(n)})=v(J_m^{(n)})$ for all $n\leq m$ and $\reg(R/J_m^{(n)})\neq v(J_m^{(n)})$ for all $n>m$?
\end{Question}	
Next we compute the local $v$-numbers of a cycle.		
\begin{Theorem}\label{cycle}
	Let $C$ be a cycle with the vertex set $V(C)=\{1,\ldots,u\}$ and the edge set
	$E(C)=\{(1,2),\ldots,({u-1},u), (u,1)\}$. Let $J$ be the cover ideal of $C$ and $\mathcal I=\{J^{(n)}\}$. Then for any $\mfp\in\overline A(\mathcal I)$ and $n\geq 1$, 
	\[\displaystyle v_{\mfp}(J^{(n)})= \left\{
	\begin{array}{l l}
	~~~~~\displaystyle\frac{|V(C)|}{2}n+c& \quad \text{if $n$ is even }\\ \vspace{0.3mm}\\
	\displaystyle\displaystyle\frac{|V(C)|}{2}n+d& \quad \text{if $n$ is odd }\\ 
	\end{array} \right.\]	where $c=d=0$ if $C$ is an even cycle and $c=0$, $d=-\frac{1}{2}$ if $C$ is an odd cycle. 
\end{Theorem}	
\begin{proof} Let $P_i=\langle x_i, x_{i+1}\rangle$ for all $1\leq i\leq u-1$ and $P_u=\langle x_u, x_1\rangle$. Then  by \cite[Theorem 3.7]{CEHH}, for all $n\geq 1$, we have $J^{(n)}=\bigcap\limits_{i=1}^u P_i^n$ and by Remark \ref{Noetherian}, $\mathcal I$ is a Noetherian filtration. For all $n\geq 1$, we have 
	$L_n:=\displaystyle \frac{J^{(n)}:P_1}{J^{(n)}}=\frac{P_1^{n-1}\cap(\bigcap\limits_{i=2}^u P_i^n)}{\bigcap\limits_{i=1}^u P_i^n}.$
	Let $\prod\limits_{i=1}^{u}x_i^{\alpha_i}+J^{(n)}$ be a nonzero element in $L_{n}$. Then
	\begin{enumerate}
		\item $\alpha_i+\alpha_{i+1}\geq n$ for all $2\leq i\leq u-1$ and $\alpha_1+\alpha_u\geq n$,
		\item $\alpha_1+\alpha_2=n-1$.
	\end{enumerate}	 
	We divide the proof in two cases. 
	
	\vspace{3mm}
	{\underline{Case 1:}}	We consider $u=2m$ for some $m\in\NN_{>0}$. 
	First we show $(L_{n})_{mn}\neq 0$ and $(L_{n})_{s}=0$ for all $s<mn$. Let $\prod\limits_{i=1}^{u}x_i^{\alpha_i}+J^{(n)}$ be a nonzero element in $L_{n}$. Then 
	$$\sum\limits_{i=1}^{u}\alpha_i=\sum\limits_{i=1}^{m-1}(\alpha_{2i}+\alpha_{2i+1})+(\alpha_1+\alpha _{2m})\geq (m-1)n+n=mn.$$ Hence $(L_{n})_{s}=0$ for all $s<mn$. Note that if $n=2k$ for some $k\in\NN_{>0}$ then 
	
		$x_1^{k}x_u^{k}(\prod\limits_{i=1}^{m-1} x_{2i}^{k-1}x_{2i+1}^{k+1})+J^{(2k)}$ is a nonzero element in $(L_{n})_{mn}$ and if $n=2k+1$ for some $k\in\NN$ then $x_1^{k}x_u^{k+1}(\prod\limits_{i=1}^{m-1} x_{2i}^{k}x_{2i+1}^{k+1})+J^{(2k+1)}$ is a nonzero element in $(L_{n})_{mn}$. 
	
	Therefore for all $n\geq 1$,  by  \cite[Lemma 1.2]{Co}, we have $ v_{P_1}({J^{(n)}})=mn=\frac{u}{2}n$. Similar calculation shows that $v_{P}({J^{(n)}})=mn=\frac{u}{2}n$ for all $n\geq 1$ and $P\in \overline A(\mathcal I)$. Hence $v(J^{(n)})= \frac{u}{2}n$ for all $n\geq 1$.
		
\vspace{3mm}
	{\underline{Case 2:}}	We consider $u=2m+1$ for some $m\in\NN_{>0}$. 
	
	First we show $(L_{2n})_{n(2m+1)}\neq 0$ and $(L_{2n})_{s}=0$ for all $s<n(2m+1)$ and  positive integer $n$. Let $\prod\limits_{i=1}^{u}x_i^{\alpha_i}+J^{(2n)}$ be a nonzero element in $(L_{2n})_s$. Suppose  $s<n(2m+1)$. Then $\sum\limits_{i=1}^u\alpha_i=s<n(2m+1)$  implies $\sum\limits_{i=3}^u\alpha_i=s-(\alpha_1+\alpha_2)=s-(2n-1)<2mn-n+1$. Hence $2n(m-1)+\alpha _{u}\leq \sum\limits_{i=1}^{m-1}(\alpha_{2i+1}+\alpha_{2i+2})+\alpha _{u}<2mn-n+1$. Hence $\alpha _{u}\leq n$.
	
	On the other hand, $ \alpha _{1}+2mn\leq \alpha _{1}+\sum\limits_{i=1}^{m}(\alpha_{2i}+\alpha_{2i+1})=s<n(2m+1)$. Thus $\alpha_1< n$. This contradicts that $\alpha_1+\alpha_u\geq 2n$. Hence $(L_{2n})_{s}=0$ for all $s<n(2m+1)$. Note that $x_1^n(\prod\limits_{i=1}^{m} x_{2i}^{n-1}x_{2i+1}^{n+1})+J^{(2n)}$ is a nonzero element in $(L_{2n})_{n(2m+1)}$.
	
	Now we show $(L_{2n+1})_{n(2m+1)+m}\neq 0$ and $(L_{2n+1})_{s}=0$ for all $s<n(2m+1)+m$ and nonnegative integer $n$. Let $\prod\limits_{i=1}^{u}x_i^{\alpha_i}+J^{(2n+1)}$ be a nonzero element in $(L_{2n+1})_s$. Suppose $s<n(2m+1)+m$. Then $\sum\limits_{i=1}^u\alpha_i=s<n(2m+1)+m$ implies $\sum\limits_{i=3}^u\alpha_i=s-(\alpha_1+\alpha_2)=s-2n<2mn-n+m$. Hence $(2n+1)(m-1)+\alpha _{u}\leq \sum\limits_{i=1}^{m-1}(\alpha_{2i+1}+\alpha_{2i+2})+\alpha _{u}<2mn-n+m$ and  $\alpha _{u}< n+1$. On the other hand, $ \alpha _{1}+m(2n+1)\leq \alpha _{1}+\sum\limits_{i=1}^{m}(\alpha_{2i}+\alpha_{2i+1})=s<n(2m+1)+m.$ Thus $\alpha_1< n$. This contradicts that $\alpha_1+\alpha_u\geq 2n+1$. Hence $(L_{2n+1})_{s}=0$ for all $s<n(2m+1)+m$.  Note that $x_1^n(\prod\limits_{i=1}^{m} x_{2i}^{n}x_{2i+1}^{n+1})+J^{(2n+1)}$ is a nonzero element in $(L_{2n+1})_{n(2m+1)+m}$. Therefore,  by  \cite[Lemma 1.2]{Co}, we get
	
	\[\displaystyle v_{P_1}({J^{(n)}})= \left\{
	\begin{array}{l l}
	~~~~~\displaystyle \frac{n(2m+1)}{2}=\frac{u}{2}n& \quad \text{if $n$ is even }\\ \vspace{0.3mm}\\
	\displaystyle\frac{(n-1)(2m+1)}{2}+m=\frac{u}{2}n-\frac{1}{2}& \quad \text{if $n$ is odd. }\\ 
	\end{array} \right.\] 
	Similar calculation shows that $v_{P}({J^{(n)}})=v_{P_1}({J^{(n)}})$ for all $n\geq 1$ and $P\in \overline A(\mathcal I)$. Hence $v(J^{(n)})= v_{P_1}({J^{(n)}})$ for all $n\geq 1$. 
\end{proof}	
\vspace{-1cm}
A vertex of a graph $G$ is called a pendant vertex if it is an endpoint of exactly one edge of $G$ and an edge of $G$ is called a pendant edge if one of its endpoints is a pendant vertex. For integers $m\geq 2$ and $s\geq 1$, by $K_m^s$ we denote a graph with $m(s+1)$ vertices which is obtained from a complete graph $K_m$ with $s$ pendant edges at each vertex of $K_m$.  
\\We denote the vertices and edges of $K_m^s$ as follows:
\begin{enumerate}
	\item $V(K_m^s)=V(K_m)\cup V$ where $V(K_m)=\{1,\ldots,m\}$ is the set of all vertices of $K_m$ and $V=\{j1,\ldots,js: 1\leq j\leq m\}$ is the set of all pendant vertices.
		\item $E(K_m^s)=\{(i,j),(j,jl): 1\leq i,j\leq m, i<j, 1\leq l\leq s\}$.
\end{enumerate}
{
\begin{figure}[H]
	\centering
	\includegraphics[width=0.55\textwidth]{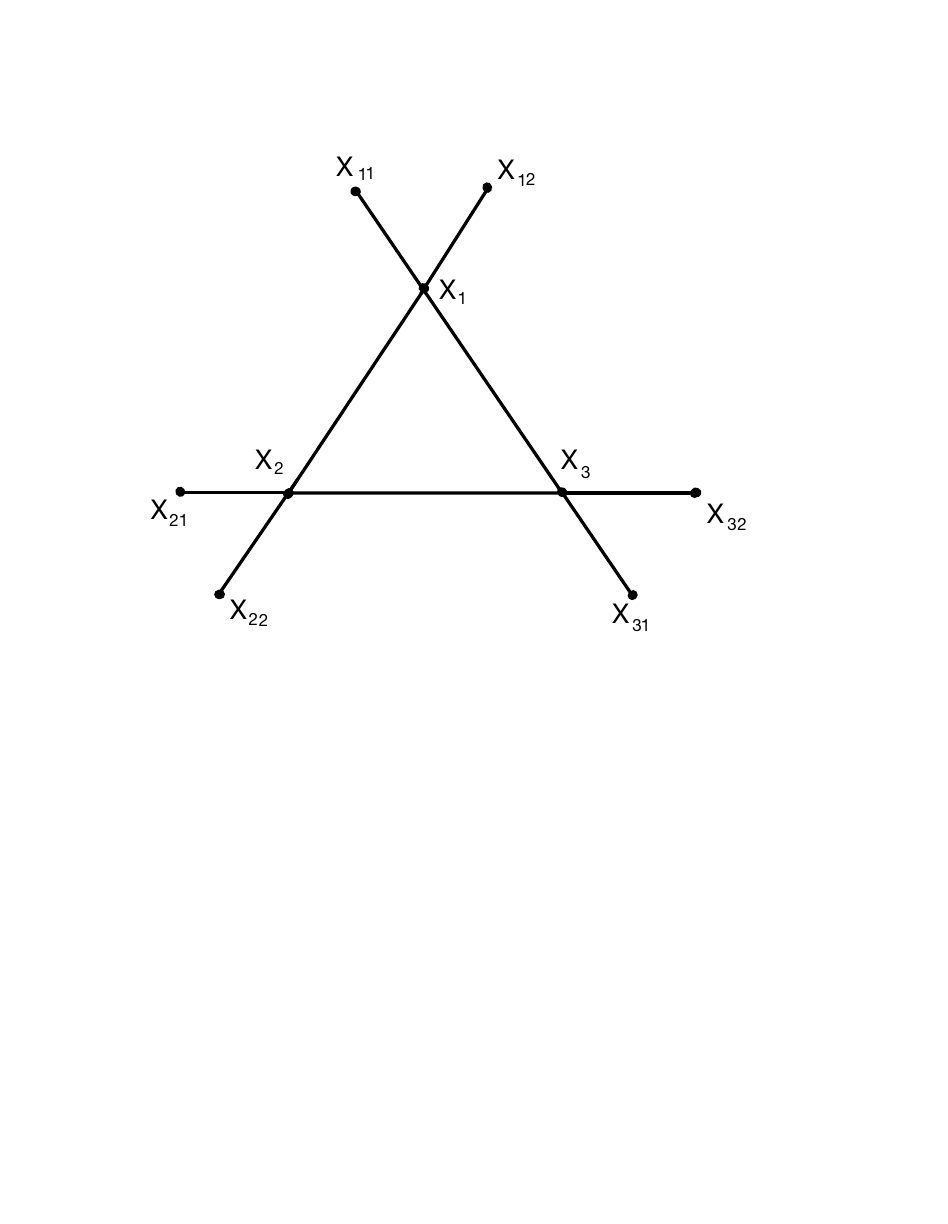}
	\vspace{-5cm}\\
	\hspace{-4cm}	\caption{$K_3^2$}	
\end{figure}
}

	\begin{Theorem}\label{corona}
	Let $J$ be the cover ideal of 	$K_m^s$ and $\mathcal I=\{J^{(n)}\}$. Consider the prime ideals  $P_{j,l}=(x_j,x_{jl})$  and $Q_{i,j}=(x_i,x_j)$ with $1\leq i,j\leq m$, $i<j$ and $1\leq l\leq s$. Then  for all $n\geq 1$, $1\leq i,j\leq m$, $i<j$ and $1\leq l\leq s$,
	\begin{enumerate}
	\item[$(1)$] $v_{P_{j,l}}({J^{(n)}})=mn+s-2$, 	
	\item[$(2)$] $v_{Q_{i,j}}({J^{(n)}})=(m+s-1)n+s-1$,
	\item[$(3)$] $v(J^{(n)})=mn+s-2$.
			\end{enumerate}		
\end{Theorem}
\begin{proof}
Note that,  by \cite[Theorem 3.7]{CEHH}, $J^{(n)}=(\bigcap\limits_{1\leq i<j\leq m}Q_{i,j}^n)\cap(\bigcap\limits_{1\leq i\leq m}(P_{i,1}^n\cap \cdots\cap P_{i,s}^n))$ and  by Remark \ref{Noetherian}, $\{J^{(n)}\}$ is a Noetherian filtration.	

For any $n\geq 1$, let $$\displaystyle L_n^{1,1}:=\frac{J^{(n)}:P_{1,1}}{J^{(n)}}=\frac{(\bigcap\limits_{1\leq i<j\leq m}Q_{i,j}^n)\cap P_{1,1}^{n-1}\cap P_{1,2}^n\cap\cdots\cap P_{1,s}^n\cap(\bigcap\limits_{2\leq i\leq m} (P_{i,1}^{n}\cap\cdots\cap P_{i,s}^n))}{(\bigcap\limits_{1\leq i<j\leq m}Q_{i,j}^n)\cap(\bigcap\limits_{1\leq i\leq m}(P_{i,1}^n\cap \cdots\cap P_{i,s}^n))}.$$
Let $\prod\limits_{1\leq i\leq m}x_i^{\alpha_i}\prod\limits_{1\leq i\leq m} x_{i1}^{\alpha_{i1}}\cdots x_{is}^{\alpha_{is}}+J^{(n)}$ be a nonzero element in 	$L_n^{1,1}$. Then
\begin{itemize}
	\item[$(1)$]	$\alpha_1+\alpha_{11}=n-1$,
	\item[$(2)$] $\alpha_1+\alpha_{1l}\geq n$ for all $2\leq l\leq s$,
	\item[$(3)$] $\alpha_j+\alpha_{jl}\geq n$ for $j\in\{2,\ldots,m\}$ and $l\in\{1,\ldots, s\}$.
	\item[$(4)$] $\alpha_i+\alpha_j\geq n$ for $i< j$ and $i,j\in\{1,\ldots,m\}$.
\end{itemize}
From $(1)$ and $(2)$, we get $\alpha_{1l}\geq 1$ for all $2\leq l\leq s$.	
Thus 
\begin{eqnarray*}
	&&\sum\limits_{i=1}^m\alpha_i+	\sum\limits_{i=1}^m(\alpha_{i1}+\cdots+\alpha_{is})\\&=& (\alpha_1+\alpha_{11})+\sum\limits_{l=2}^s\alpha_{1l}+ \sum\limits_{i=2}^m(\alpha_i+\alpha_{i1})+\sum\limits_{i=2}^m(\alpha_{i2}+\cdots+\alpha_{is})\\&\geq& n-1+(s-1)+(m-1)n+	\sum\limits_{i=2}^m(\alpha_{i2}+\cdots+\alpha_{is})\\&\geq& mn+s-2.
\end{eqnarray*} 	

Note that $x_1^{n-1}x_{12}\cdots x_{1s}\prod\limits_{j=2}^mx_j^n+J^{(n)}$ is a non zero element in $L_n^{1,1}$. Therefore,  by  \cite[Lemma 1.2]{Co}, we get $v_{P_{1,1}}(J^{(n)})=mn+s-2$ for all $n\geq 1$. Similar argument shows $v_{P_{i,l}}(J^{(n)})=mn+s-2$ for all $n\geq 1$, $i\in\{1,\ldots,m\}$ and $l\in\{1,\ldots,s\}$.

For any $n\geq 1$ , let  $$F_n^{1,2}:=\frac{J^{(n)}:Q_{1,2}}{J^{(n)}}=\displaystyle\frac{Q_{1,2}^{n-1}\cap \bigcap\limits_{\substack{1\leq i<j\leq m\\(i,j)\neq (1,2)}} Q_{i,j}^n\cap(\bigcap\limits_{1\leq i\leq m}(P_{i,1}^n\cap \cdots\cap P_{i,s}^n))}{(\bigcap\limits_{1\leq i<j\leq m}Q_{i,j}^n)\cap(\bigcap\limits_{1\leq i\leq m}(P_{i,1}^n\cap \cdots\cap P_{i,s}^n))}.$$

Let $\prod\limits_{1\leq i\leq m}x_i^{\alpha_i}\prod\limits_{1\leq i\leq m} x_{i1}^{\alpha_{i1}}\cdots x_{is}^{\alpha_{is}}+J^{(n)}$ be a nonzero element in 	$F_n^{1,2}$. Then
\begin{itemize}
	\item[$(1)$]	$\alpha_1+\alpha_{2}=n-1$,
	\item[$(2)$] $\alpha_i+\alpha_{j}\geq n$ for all $i,j\in\{1,\ldots,m\}$, $i<j$ and $(i,j)\neq (1,2)$,
	\item[$(3)$] $\alpha_j+\alpha_{jl}\geq n$ for $j\in\{1,\ldots,m\}$ and $l\in\{1,\ldots,s\}$.	
\end{itemize}
From $(1)$ and $(3)$, we get 
$s(n-1)+\sum\limits_{l=1}^s(\alpha_{1l}+\alpha_{2l})=\sum\limits_{i=1}^2\sum\limits_{l=1}^s(\alpha_{i}+\alpha_{il})\geq 2sn.$
 Hence $\sum\limits_{i=1}^2\sum\limits_{l=1}^s\alpha_{il}\geq sn+s$.	
Then 
\begin{eqnarray*}
&&\sum\limits_{i=1}^m\alpha_i+	\sum\limits_{i=1}^m(\alpha_{i1}+\cdots+\alpha_{is})\\&=& (\alpha_1+\alpha_2)+\sum\limits_{i=1}^2\sum\limits_{l=1}^s\alpha_{il}+\sum\limits_{j=3}^m(\alpha_j+\alpha_{j1})+\sum\limits_{j=3}^m\sum\limits_{l=2}^s\alpha_{jl}
\\&\geq& n-1+sn+s+(m-2)n+\sum\limits_{j=3}^m\sum\limits_{l=2}^s\alpha_{jl}
\\&\geq& (m+s-1)n+s-1.
\end{eqnarray*}

{\underline{Case 1:}} Let $n=2k$ for some positive integer $k$. Note that $x_1^kx_2^{k-1}\prod\limits_{1\leq l\leq s} x_{1l}^{k}\prod\limits_{1\leq l\leq s} x_{2l}^{k+1}\prod\limits_{j=3}^m x_{j}^{2k}+J^{(n)}$  is a non zero element in $F_n^{1,2}$.

{\underline{Case 2:}} Let $n=2k+1$ for some nonnegative integer $k$. Note that

$x_1^kx_2^{k}\prod\limits_{1\leq l\leq s} x_{1l}^{k+1}\prod\limits_{1\leq l\leq s} x_{2l}^{k+1}\prod\limits_{j=3}^m x_{j}^{2k+1}+J^{(n)}$  is a non zero element in $F_n^{1,2}$.

Therefore,  by  \cite[Lemma 1.2]{Co}, we have $v_{Q_{1,2}}(J^{(n)})=n(m+s-1)+s-1$ for all $n\geq 1$. A similar argument shows $v_{Q_{i,j}}(J^{(n)})=n(m+s-1)+s-1$ for all $n\geq 1$ and $1\leq i<j\leq m$. 
\end{proof}
As a consequence of the above result we obtain that for Noetherian filtrations the difference between the leading coefficients of local $v$-numbers, i.e., $\displaystyle\lim\limits_{n\to \infty}\frac{v_Q(I_n)}{n}- \lim\limits_{n\to \infty}\frac{v_P(I_n)}{n}$ can be arbitrary large and this helps us to provide a counterexample to \cite[Conjecture 5.4]{FSmarch}.
\begin{Corollary}\label{counterexample}
	For every positive integer $t$,
	\begin{enumerate}
		\item[$(1)$] 	there exists a Noetherian filtration $\{I_n\}$ in a Noetherian $\mathbb  N$-graded domain such that $\Ass(I_n)=\Ass(I_{n+1})$ for all $n\gg0$ and there exist two primes $P,Q$ which are maximal in $\Ass(I_n)$ for all $n\gg0$ such that $\displaystyle\lim\limits_{n\to \infty}\frac{v_Q(I_n)}{n}- \lim\limits_{n\to \infty}\frac{v_P(I_n)}{n}=t$.	
		\item[$(2)$]	there exists a homogeneous ideal $I_t$ in a Noetherian $\mathbb  N$-graded domain and there exist two primes $P,Q$ which are maximal in $\overline A(I_t)$  such that $\displaystyle\lim\limits_{n\to \infty}\frac{v_P(I_t^n)}{n}\neq \lim\limits_{n\to \infty}\frac{v_Q(I_t^n)}{n}$.
	\end{enumerate}			
\end{Corollary}	
\begin{proof}
$(1)$ Let $R=K[x_1,\ldots,x_{m(t+2)}]$ with $m\geq 2$, $t\geq 1$ and $J$ be the cover ideal of $K_m^{t+1}$. By Remark \ref{Noetherian}, $\{I_n=J^{(n)}\}$ is a Noetherian filtration. 

Let $P=P_{1,1}$ and $Q=Q_{1,2}$ where $P_{1,1}$ and $Q_{1,2}$
	are as in Theorem \ref{corona}. Then
$$\lim\limits_{n\to \infty}\frac{v_Q(I_n)}{n}- \lim\limits_{n\to \infty}\frac{v_P(I_n)}{n}=m+(t+1)-1-m=t.$$

$(2)$ Let $R=K[x_1,\ldots,x_{m(s+1)}]$ with $m\geq 2$, $s\geq 2$ and $J$ be the cover ideal of $K_m^s$. By Remark \ref{Noetherian}, $\{J^{(n)}\}$ is a Noetherian filtration. By \cite[Theorem 5.1]{HHT}, $J^{(2n)}=(J^{(2)})^n$ for all $n\geq 1$. 

For any positive integer $t$, take $I_t=J^{(2t)}$. Note that $I_t$ has no embedded primes. Let $P=P_{1,1}$ and $Q=Q_{1,2}$ where $P_{1,1}$ and $Q_{1,2}$
are as Theorem \ref{corona}.

Then by Theorem \ref{corona}, we get,
$$\lim\limits_{n\to \infty}\frac{v_P(I_t^n)}{n}=\lim\limits_{n\to \infty}\frac{v_P((J^{(2t)})^n)}{n}=\lim\limits_{n\to \infty}\frac{v_P(J^{(2tn)})}{n}=2tm\mbox{ and }$$
$$\lim\limits_{n\to \infty}\frac{v_Q(I_t^n)}{n}=\lim\limits_{n\to \infty}\frac{v_Q((J^{(2t)})^n)}{n}=\lim\limits_{n\to \infty}\frac{v_Q(J^{(2tn)})}{n}=2t(m+s-1).$$
\end{proof}		
 
\subsection*{Acknowledgements.}
The authors would like to thank Kamalesh Saha for his helpful comments. 


\begin{thebibliography}{1000000000}
		
		\bibitem{AM}  M. F. Atiyah and I. G. Macdonald, Introduction to commutative algebra,  Addison-Wesley Ser. Math.

	
	
	\bibitem{BMS} P. Biswas, M. Mandal, K. Saha,  Asymptotic behaviour and stability index of v-numbers of graded ideals, arXiv:2402.16583.
	
	
	\bibitem{Bo} N. Bourbaki,  Commutative Algebra, Chapters 1-7, Springer Verlag, 1989.
	
	\bibitem{Br} M. Brodmann, Asymptotic stability of $\mbox{Ass}(M/I^nM)$ Proc. Amer. Math. Soc. 74 (1979), 16-18.
	
	
	\bibitem{Co} A. Conca, A note on the v-invariant, to appear in Proc. Amer. Math. Soc.
	
	\bibitem{CEHH} S. M.  Cooper, R. J. D. Embree, H. T. Ha and A. H.  Hoefel, Symbolic powers of monomial ideals , Proc. Edinb. Math. Soc. (2) 60 no. 1, (2017),  39-55.
	
	\bibitem{CSTPV} S. M. Cooper, A. Seceleanu, S. O. Tohaneanu, M. Vaz Pinto and R. H. Villarreal, Generalized minimum distance functions and algebraic invariants of Geramita ideals, Adv. in Appl. Math. (2020), 101940.
	
	\bibitem{FS2018} S. A. Seyed Fakhari, Symbolic powers of cover ideal of very well-covered and bipartite graphs, Proc. Amer. Math. Soc. 146, 97-110 (2018).
	
	\bibitem{Fa} S. A. Seyed Fakhari, Regularity of symbolic powers of cover ideals of graphs, Collect. Math. 70, no. 2, (2019),  187-195.
	
	
	
	\bibitem{FS} A. Ficarra and E. Sgroi, Asymptotic behaviour of the v-number of homogeneous ideals, arXiv:2306.14243. 
	
	\bibitem{F23} A. Ficarra, Simon conjecture and the $v$-number of monomial ideals, to appear in Collect. Math.


	\bibitem{FSmarch} A. Ficarra and E. Sgroi, Asymptotic behaviour of integer programming and the v-function of a graded filtration, arXiv:2403.08435.
	
	\bibitem{GF} L. Fiorindo and D. Ghosh, On the asymptotic behaviour of the Vasconcelos invariant for graded modules, 2024, preprint arXiv:2401.16358.
	
	
	\bibitem{GRV1} I. Gitler, E. Reyes and R. H. Villarreal, Blowup algebras of ideals of vertex covers of bipartite graphs, Contemp. Math., 376 American Mathematical Society, Providence, RI, 2005, 273-279.
	
	
	\bibitem{GS} E. Grifo and A. Seceleanu, Symbolic Rees algebras, Springer, Cham, 2021, 343–371.
	
	\bibitem{GRV} G. Grisalde, E. Reyes and R. H. Villarreal, Induced matchings and the v-number of graded ideals. Mathematics, 9 (22), 2021.
	
	\bibitem{HHT} J. Herzog, T. Hibi, N. V.  Trung, Symbolic powers of monomial ideals and vertex cover algebras, Adv. Math. 210, no.1, (2007),  304-322.
	
	\bibitem{HH} J. Herzog and T. Hibi, Monomial ideals, Graduate Texts in Mathematics, 260, Springer-Verlag London, Ltd., London, (2011).
	
	\bibitem{HJKN} H. T. H\`{a},  A. V. Jayanthan, A. Kumar and H. D. Nguyen, Binomial expansion for saturated and symbolic powers of sums of ideals, J. Algebra 620 (2023), 690-710.
	
	\bibitem{JV} D. Jaramillo and R. H. Villarreal, The v-number of edge ideals, J. Combinatorial Theory A (2021), no. 177. 
	
	\bibitem{JVS} D. Jaramillo-Velez and L. Seccia, Connected domination in graphs and v-numbers of binomial edge ideals, Collect. Math. (2023).
	
	

	\bibitem{Mat} H. Matsumura, Commutative algebra Matsumura,  Math. Lecture Note Ser., 56.
	
	\bibitem{KKSS} A. Kumar, R. Kumar, R. Sarkar and S. Selvaraja, Symbolic powers of certain cover ideals of graphs, Acta Math. Vietnam. 46 (2021), no. 3, 599-611.
	
	\bibitem{MRK} M. Kumar, R. Nanduri and K. Saha, The slope of v-function and Waldschmidt Constant, arXiv:2404.00493.
	
		
	\bibitem{Rat} J. L. Ratliff Jr., On prime divisors of  $I^n$,   $n$  large, Michigan Math. J. 23, no. 4, (1977), 337–352 .
	
	\bibitem{Rat2} J. L. Ratliff Jr., On asymptotic prime divisors, Pacific J. Math. 111, no. 2, (1984),  395–413.
	
		\bibitem{KSa} K. Saha, The v-Number and Castelnuovo-Mumford Regularity of Cover Ideals of Graphs, Int. Math. Res. Not. IMRN (2024), no. 11, 9010–9019. 
	
	\bibitem{SG} J. G. Serio, Multiplicities in Commutative Algebra, ProQuest LLC, Ann Arbor, MI, (2016), 134 pp.
	
	\bibitem{SH} I. Swanson and C. Huneke, Integral Closure of Ideals, Rings and Modules, Cambridge University Press, 2006.
	
	\bibitem {SS} K. Saha and I. Sengupta, The v-number of monomial ideals, J. Algebraic Combin. 56 (2022), no. 3, 903–927.
 
 \bibitem{Vil} R. H. Villarreal, Unmixed bipartite graphs, Rev. Colombiana Mat. 41 (2007), no. 2, 393-395.
	
	
\end{thebibliography}
	\end{document}